\documentclass[oneside,11pt,reqno]{amsart}
\usepackage{hyperref,amssymb,amsmath,amsthm,bbm,enumerate,mdwlist,url,multirow,stmaryrd}
\usepackage[pdftex]{graphicx}
\usepackage[shortlabels]{enumitem}
\usepackage{color}

\usepackage{xcolor}

\theoremstyle{definition}
\newtheorem{definition}{Definition}

\theoremstyle{theorem}
\newtheorem{proposition}[definition]{Proposition}
\newtheorem{corollary}[definition]{Corollary}
\newtheorem{lemma}[definition]{Lemma}
\newtheorem{theorem}[definition]{Theorem}
\newtheorem{assumption}[definition]{Assumption}
\newtheorem{example}[definition]{Example}

\numberwithin{equation}{section}
\numberwithin{definition}{section}
\theoremstyle{remark}
\newtheorem{remark}[definition]{Remark}

\newtheorem{qn}[definition]{Question}

\def\bip{\null}

\def\half{\frac{1}{2}}
\def\defto{{\eqdef }}
\def\esssup{\mathrm{esssup}}
\def\essinf{\mathrm{essinf}}
\def\PPesssup{\PP\mbox{-}\esssup}
\def\PP{\mathsf P}
\def\P{\mathbb P}
\def\QQ{\mathsf Q}
\def\Q{\mathbb Q}
\def\AA{\mathcal A}
\def\EE{\mathsf E}
\def\EEE{\mathcal E}
\def\E{\EE}
\def\RR{\mathbb R}
\def\R{\RR}
\def\Epsilon{\mathcal E}
\def\EEc{\EE^{\PP^c}}
\def\FF{\mathcal F}
\def\NN{\mathcal N}
\def\HH{\mathcal H}
\def\CC{\mathbf C}
\def\defto{\buildrel {\mathrm{def}}\over =}
\def\lto{\buildrel {L^1}\over \rightarrow}
\def\eqdef{\defto}
\def\nto{\buildrel {n\rightarrow\infty}\over \rightarrow}
 
\def\GG{\mathcal G}
\def\G{\GG}
\def\S{S}

\def\SS{\mathcal{S}}
\def\BB{\mathcal{B}}
\def\TT{\mathcal{T}}
\def\DD{\mathcal D}
\def\LL{\mathcal L}
\def\Im{\mathrm{Im}}
\def\Gtimes{\mathbf{G}}
\def\bG{\Gtimes}
\def\Exp{\mathrm{Exp}}
\def\Unif{\mathrm{Unif}}

\def\stop{{W}}
\def\1{\mathbbm{1}}

\newcommand\xqed[1]{%
  \leavevmode\unskip\penalty9999 \hbox{}\nobreak\hfill
  \quad\hbox{#1}}
\newcommand\finish{\xqed{$\diamond$}}

\def\esssup{\mathrm{esssup}}
\def\essinf{\mathrm{essinf}}
\def\PPesssup{\PP\mbox{-}\esssup}
\def\PP{\mathsf P}
\def\P{\mathbb P}
\def\QQ{\mathsf Q}
\def\AA{\mathcal A}
\def\EE{\mathsf E}
\def\RR{\mathbb R}
\def\Epsilon{\mathcal E}
\def\EEc{\EE^{\PP^c}}
\def\FF{\mathcal F}
\def\HHH{\mathcal R}
\def\NN{\mathcal N}
\def\HH{\mathcal H}
\def\CC{\mathbf C}
\def\GG{\mathcal G}
\def\SS{\mathcal{S}}
\def\BB{\mathcal{B}}
\def\TT{\mathcal{T}}
\def\DD{\mathcal D}
\def\LL{\mathcal L}
\def\Im{\mathrm{Im}}
\def\Gtimes{\mathbf{G}}
\def\Exp{\mathrm{Exp}}
\def\Unif{\mathrm{Unif}}

\def\id{\mathrm{id}}
\usepackage{xcolor}

\newsavebox\MBox

\begin{document}
\title{On the informational structure in optimal dynamic stochastic control}
\author{Saul Jacka}
\address{Department of Statistics, University of Warwick, UK}
\address{The Alan Turing Institute, London, UK}
\email{s.d.jacka@warwick.ac.uk}
\author{Matija Vidmar}
\address{Department of Mathematics, University of Ljubljana, Slovenia}
\address{Institute for Mathematics, Physics and Mechanics, Ljubljana, Slovenia}
\email{matija.vidmar@fmf.uni-lj.si}
\begin{abstract}
We formulate a very general framework for optimal dynamic stochastic control problems which allows for a control-dependent informational structure. The issue of informational consistency is investigated. Bellman's principle is formulated and proved. In a series of related results, we expound on the informational structure in the context of (completed) natural filtrations of stochastic processes.
\end{abstract}
\thanks{This research was initiated while MV was a PhD student in the Department of Statistics at the University of Warwick, and a recipient of a scholarship from the Slovene Human Resources Development and Scholarship Fund (contract number 11010-543/2011). MV acknowledges this latter support and also that of the Slovenian Research Agency (research core funding No. P1-0222); he thanks the hospitality of the former. We thank Jon Warren, for many helpful discussions on some of the material from section~\ref{appendix:stopped} of this paper.}



\maketitle

\section{Introduction}
\subsection{Motivation and overview of results}
Optimal dynamic stochastic control, is \emph{stochastic}, in the sense that the output of the system is random; \emph{optimal} in that the goal is optimization of some expectation; and it is  \emph{dynamic}, in that the control is adapted to the current and past state of the system. In general, however, the controller in a dynamic stochastic control problem can observe only some of the information which is being accumulated; and what is observed may vary,  depending on the chosen control. 

It is reasonable, therefore, to insist that all admissible controls (when regarded as processes) are adapted (or even predictable with respect) to a \emph{control-dependent-information-flow} that is being acquired by the controller.

Some (informal) examples follow.

\begin{example}\emph{Job hunting}. Consider an employee trying to optimize her choice of employer. The decision of whether or not, and where to move, will be based on the knowledge of the qualities of the various potential employers, this knowledge itself depending on her previous employments. 
\end{example}
\begin{example}\emph{Quality control in a widget production line}. A fault with the line may cause all the widgets from some time onwards to be faulty. Once this has been detected, the line can be fixed and the production of functioning widgets restored. However, only the condition of those widgets which are taken out of the production line, and tested, can actually be observed. A cost is associated with this process of testing. Conversely, rewards accrue from producing functioning widgets.  
\end{example}
\begin{example}\emph{Tanaka's SDE}.  From the theory of controlled SDEs, a classic example of \emph{loss of information} is Tanaka's SDE: let $X$ be a Brownian motion and $W_t\eqdef \int_0^t\mathrm{sgn}(X_s)dX_s$, $t\in [0,\infty)$. Then the completed natural filtration of $W$ is strictly included in the completed natural filtration of $X$ \cite[p. 302]{karatzasshreve}. 
\end{example}
\begin{example}\emph{Bandit models}.  The class of \emph{bandit} models is well-known (see the excellent literature overview in \cite{fryer}). These are ``sequential decision problems where, at each stage, a resource like time, effort or money has to be allocated strategically between several options, referred to as the arms of the bandit \ldots The key idea in this class of models is that agents face a trade-off between experimentation (gathering information on the returns to each arm) and exploitation (choosing the arm with the highest expected value).'' \cite[p. 2]{fryer}. 
\end{example}
Other situations in which the control is non-trivially, and naturally, adapted or predictable with respect to an information-flow, which it both influences and helps generate, abound; for example, controlling the movement of a probe in a stochastic field, only the local values of the field being observable  (see Example~\ref{ex:cts_field}); a similar situation for movement on a random graph, the controller only being able to observe the values attached to the vertices visited; additively controlling a process, but observing the sum of the process itself and of the control.

Analysis of dynamic stochastic control problems is vastly facilitated by the application of Bellman's principle. Informally this states that (i) the best control is to behave optimally now, conditionally on behaving optimally in the future, and that (ii) the longer one pursues a fixed, but arbitrary, policy before behaving optimally the worse one's payoff from the control problem.

Failure to reflect the control-dependent informational flow explicitly in the filtration structure, may result in Bellman's super(martingale) principle not being valid. This is exemplified in the following:

\begin{example}
\emph{Box picking.}  Let $\Omega\eqdef \{0,1\}\times\{-1,1\}$, endow it with the discrete $\sigma$-field $\HH\eqdef 2^\Omega$ and the probability measure $\PP$, given by $\PP(\{(0,-1)\})=1/6$, $\PP(\{(0,1)\})=1/3$, $\PP(\{(1,-1)\})=1/3$, $\PP(\{(1,1)\})=1/6$.  
Let $Y_1$ and $Y_2$ be the projections onto the first and second coordinate, 
respectively; set $X_0\eqdef 0$.  The time set is $\{0,1,2\}$. Note that $Y_1$ stochastically dominates $Y_2$ and that on $\{Y_1=1\}$, $\{Y_2\leq Y_1\}$, whilst conditionally on $\{Y_1=0\}$, $Y_2$ has a strictly positive conditional expectation. 

The set of admissible controls are processes $c=(c_i)_{i=0}^2$ such that $c_0$ is 0, while $c_1$ and $c_2$ are $\{1,2\}$-valued with $c$ predictable with respect to the natural filtration of the controlled process $X^c$ which is given by $X^c_t\defto Y_{c_t}$ for $t\in \{0,1,2\}$. We can think of  being allowed to open sequentially either of two boxes containing rewards.
 
So, we decide on index $1$ or $2$ and then after observing the corresponding $Y_1$ or $Y_2$, decide again on $1$ or $2$. The objective functional, whose expectation is to be maximized, is $J(c)\eqdef X^c_1+X^{c}_2$ for admissible controls $c$. 

For an admissible control $c$ and $t\in \{0,1,2\}$, let $D(c,t)$ be the set of admissible controls that agree with $c$ up to, and inclusive of, time $t$. 

It is then  elementary to verify that:

 \textbf{(1)} The control $c^*$ given by $c^*_0=0$, $c^*_1=1$, $c^*_2=2\1_{\{X_1=0\}}+\1_{\{X_1=1\}}$, is the unique optimal control. 

 \textbf{(2)} Setting $\FF_0\eqdef \{\emptyset,\Omega\}$ and $\FF_1\eqdef \FF_2\eqdef 2^\Omega$, $\FF=(\FF_t)_{t=0}^2$ is the smallest filtration with respect to which all the $X^c$ (as $c$ runs over the admissible controls) are adapted. The ``\textbf{classical Bellman process}" $W^*=(W^*_t)_{t=0}^2$ associated to this filtration, i.e. the process
$$W^*_t\eqdef \sup_{d\in D(c^*,t)}\EE^\PP[J(d)\vert \FF_t]\text{ for }t\in \{0,1,2\},
$$
\textbf{is not an $\FF$-supermartingale}--contradicting both parts of Bellman's principle.

 \textbf{(3)} Letting $\GG^{c}$ be the natural filtration of $X^{c}$, the 
``\textbf{informationally-consistent Bellman process}'' $\hat V^c=(\hat V^c_t)_{t=0}^2$ 
associated to this filtration, i.e. the process
$$
\hat V^c_t\eqdef \sup_{d\in D(c,t)}\EE^\PP[J(d)\vert \GG^{c}_t]\text{ for }t\in \{0,1,2\},
$$
 \textbf{is a $\GG^{c}$-supermartingale} for each admissible control $c$--confirming that this version satisfies the second part of Bellman's principle; and $ \hat V^{c^*}$ \textbf{is a $\GG^{c^*}$-martingale}--confirming that this version satisfies the first part of Bellman's principle. 

Item \textbf{(3)} also follows from  Theorems~\ref{prop:lattice} and~\ref{Bell1}  below.\finish
\end{example}
So, one should like a \emph{general} framework for stochastic control,  equipped with a suitable abstract version of Bellman's principle, which makes the  control-dependent informational flows explicit and inherent in its machinery. The circularity of the requirement that the controls are adapted to an informational structure which they themselves help generate, makes this a delicate point. 

In this paper, then, 
we describe a general (single controller, perfect recall)  stochastic control framework, which explicitly allows for a control-dependent informational flow, and  provide  (under a technical 
condition, Assumption~\ref{assumption:lattice}), a fully general, abstract version of Bellman's principle. This is the content of sections 3-6. 
Specifically, section~\ref{section:setting} formally 
defines a system of stochastic control (in which observed information is an explicit function of control); 
section~\ref{section:cond_payoff_sytem} discusses its conditional payoff and `Bellman' system; 
section~\ref{section:bellman} formulates the relevant version of Bellman's principle -- Theorem~\ref{Bell1} is our 
main result. We should emphasise that for this part of the paper we follow on from the approach of El Karoui \cite{elkaroui}.

Section~\ref{section:formal_example} contains a detailed solution of a concrete class of  examples in some reasonably straightforward cases, illustrating some of the main ideas of this paper. Several other examples and counterexamples are also given along the way. 

A crucial requirement for the programme outlined above to be successful is that of informational 
consistency over controls (see Definition~\ref{assumption:temporalconsistency}): if two controls agree 
up to a certain time, then what we have observed up to that time should also agree. Especially at the 
level of stopping times, this becomes a non-trivial statement -- for example, when the 
observed information is that generated by a controlled process, which is often the case. We explore this issue of informational consistency in the context of (completed) natural filtrations of processes in 
section~\ref{appendix:stopped}. Specifically, we consider there the following 
question, which is interesting in its own right: 
\begin{qn}\label{FiltQ} Suppose that $X$ and $Y$ are two processes, 
that $S$ is a stopping time of one or both of their (possibly completed) natural filtrations; and that the stopped 
processes $X^S$ and $Y^S$ agree (possibly only with probability one): must the two (completed) natural 
filtrations also agree \emph{at} the time $S$? 
\end{qn}
The answer to this question is non-trivial in a
continuous-time setting, and several related findings are obtained along the way (see the 
introductory remarks to section~\ref{appendix:stopped}, on p.~\pageref{appendix:stopped}, for a more 
detailed account). The main result here is Theorem~\ref{theorem:galmarino}. 

\subsection{Literature overview}
The phenomenon  of control-dependent information has, by now, entered and been studied in the stochastic control literature in and through numerous more or less \emph{specific} situations and problems, see e.g. \cite{elkaroui_jeanblanc,fryer,fleming_pardoux,wonham,davis-varaiya,varaiya} \cite[Chapter~8]{bensoussan} \cite[Sections~VI.10-11]{fleming-rishel} \cite[Subsection~2.7.6]{yong} \cite[Section~1.4]{gihman} \cite[Sections~VI.2-3]{alain} and the references therein --- the focus having been, for the most part, on reducing the original control problem, which is based on partial control-dependent observation, to an associated `separated' problem, which is based on complete observation. 

When it comes to \emph{general} frameworks for dynamic stochastic control, however, thus far, by and large, only a single, non-control dependent, observable \cite{elkaroui} informational flow \cite{striebel} \cite[Sections~2 and~3]{soner} appears to have been allowed.  Two exceptions to this were pointed out to us after this work was essentially completed. The first is the monograph of Y{\"u}ksel and Ba\c{s}ar \cite{yuksel}. It contains a framework for networked stochastic optimization in discrete time under information constraints with multiple agents and non-perfect recall \cite[Section~2.4]{yuksel}. The monograph provides an in-depth analysis of the information structures for stochastic teams \cite[Chapter~3]{yuksel}, but it does not (for its non-perfect-recall and multiple-agent generality, presumably even cannot) offer an explicit general (Bellman (super)martingale) optimality principle. The second is the paper by Rishel \cite{rishel}. Therein the author considers a stochastic control problem involving a general two-component controlled process $x_u=(y_u,z_u)$ ($u$ being the control), only the component $y_u$ being observable (and this at certain exogenously given deterministic time instances up to an endogenous finite random lifetime $\eta_u$). The payoff is the integral $\int_0^{\eta_u}k_u(s)ds$, with $k_u$ adapted to $x_u$ and satisfying certain moment integrability assumptions. The validity of  a principle of optimality \cite[Eq.~(28)]{rishel} (at deterministic times) is proved under a `relative completeness assumption' \cite[Eq.~(32)]{rishel} (cf. Assumption~\ref{assumption:lattice}). 

Our work complements and extends \cite{yuksel,rishel} in several important directions: by contrast to \cite{yuksel} (i)  Bellman's principle \emph{is} formulated and proved in the control-dependent filtration setting, both in discrete and continuous time; further, and by contrast to \cite{rishel}, this is done (ii) using the clear and appealing notion of a (super)martingale system (cf. \cite[Eq.~(28)]{rishel}); with (iii) the (super)martingale property  stated not only for deterministic times, but also for (what are extensions to the control-dependent setting of the concept of) stopping times; finally (iv) the setting of \cite{rishel} is considerably generalised and its 
assumptions are weakened.

As concerns  Question~\ref{FiltQ}, Theorem~\ref{theorem:galmarino} provides a generalization of a part of Galmarino's test, available in the literature for coordinate 
processes on canonical spaces \cite[p. 149, Theorem~IV.100]{dellacheriemeyer} \cite[p. 320, 
Lemma~4.18]{karatzasshreve}, and extended here to a non-canonical setting. 
 In particular, in existing work one finds argued, under reasonably innocuous conditions (e.g. \cite[p. 9, 
Proposition~2.18]{karatzasshreve}), that the $\sigma$-field generated by the stopped process is 
included in the history of the natural filtration of the process up to that stopping time. The opposite 
inclusion, has (to the best of our knowledge) until now only been established for coordinate 
processes on canonical spaces \cite[p.~33, Lemma~1.3.3]{varadhan} \cite{blumenthal-getoor} (the result in \cite{blumenthal-getoor} is with completions, in a strong Markov context) or, slightly more generally, under the condition that all the stopped paths of the process are already paths of the process \cite[Paragraph 1.1.3, p.~10, Theorem~6]{shiryaev}. We show in Theorem~\ref{theorem:galmarino} that it holds true far more generally.

\subsection{A convention}\label{subsection:convention}
As stated earlier, we will provide and analyze a framework for optimal dynamic stochastic control in which information is explicitly control-dependent. The informational flow  is modeled using filtrations, and this can be done in one of the following two, essentially different, ways: 
\begin{enumerate}
 \item Dealing with events `with certainty', irrespective of the presence of probability.
 \item Dealing with events `up to a.s. equality', insisting that the filtrations be complete relative to the underlying probability measure(s).
\end{enumerate}
\emph{In sections~\ref{section:setting}--\ref{section:bellman}} we will develop the second `probabilistic' approach -- of complete filtrations -- \emph{in parallel}  to the first -- `measure-theoretic' -- setting. The formal differences between the two approaches are minor. For the most part one has merely to add, in the `complete' setting, a number of a.s. qualifiers. We will put these, and any other eventual differences of the second approach as compared to the first, in \{\} braces. This will enforce a strict separation between the two settings, while allowing us to repeat ourselves as little as possible.

\subsection{Some general notation}
Throughout this paper, for a probability measure $\PP$ on $\Omega$ and 
$A\subset \Omega$, a property will be said to hold $\PP$-a.s. on $A$, if the set 
of $\omega\in A$ for which the property does not hold is first measurable (i.e. belongs to the domain of 
$\PP$), and second is of $\PP$-measure zero. When $A=\Omega$, we shall just say that the 
property holds $\PP$-a.s. 

For $\LL$ a collection of subsets of $\Omega$, 
\begin{enumerate}
\item $\sigma(\LL)=\sigma_\Omega(\LL)$ denotes the $\sigma$-field generated by $\LL$ (on $\Omega$); 
\item if  $A\subset \Omega$, 
$\LL\vert_A\eqdef \{L\cap A: L\in \LL\}$ denotes
 the trace of $\LL$ on $A$.
\suspend{enumerate}

Finally,
\resume{enumerate}
\item given a map $f$ from $\Omega$ into some measurable space $(E,\Epsilon)$, $\sigma(f)\eqdef \{f^{-1}(A):A\in \Epsilon\}$ is the $\sigma$-field generated by $f$ (the space  $(E,\Epsilon)$ being understood from context);
\item  $2^X$ denotes the power set of a set $X$.  
\end{enumerate}
 
\section{Two examples}
We start with two key examples which we shall revisit and formalise later. 
\begin{example}\label{bms}
\emph{Switching between two Brownian motions.} 
We may observe each of two independent Brownian motions (BMs), $B^0$ and $B^1$, but at any given time we may only observe one of them.\  We acquire a running reward of the difference between the BM we currently observe and the unobserved one.  At any time we may pay a state-dependent cost $K$ to switch our observation to the other BM.

The control process, $c$, gives the index of the BM we choose to observe, and is any c\`adl\`ag process taking values in $\{0,1\}$, predictable with respect  to the filtration $\GG^c_t\defto \sigma(B^{c_s}_s:\; s\leq t)$ (and satisfying certain restrictions on the time between jumps that we will specify later on p.~\pageref{ex-BMs-contd}). 

We denote by $\sigma^c(t)$ the last time that we changed our observed BM i.e. the last jump time of $c$ before time t, and $\tau^c(t)$ is the lag since the last jump i.e. $\tau^c(t)=t-\sigma^c(t)$. Then we define $Z^c$ as follows:
$$
Z^c_t\defto B^c_t-B^{1-c}_{\sigma^c(t)},
$$
so that $Z^c_t$ is the conditional mean of $B^c_t-B^{1-c}_t$ given observations up to time $t$.

The reward $J$ (which we seek to maximise) is given by
$$
J(c)=\int_0^\infty e^{-\alpha t}Z^c_tdt-\int_0^\infty e^{-\alpha t}K(Z^c_{t-},\tau(t-))|dc|_t.
$$
We formalise this example  on p.~\pageref{ex-BMs-contd}, generalise it in Example~\ref{ex:cts_field}, and solve it in some special cases in section~\ref{section:formal_example}.
\end{example}

\begin{example}\label{rand}
\emph{Poisson random measure search model.} 
The controlled process $X$ is a process in $\R^n$.
The control $c$ is the drift of $X$ and is bounded in norm by 1 so
$$
X^c_t=\int_0^tc_sds+\sigma W_t,
$$
where $W$ is a BM and $||c_t||\leq 1$.

There is an underlying locally finite 
Poisson random measure on $\R^n$, $\mu$.

We (progressively) observe $W$ and the restriction of $\mu$ to the path traced out by $B(X^c_\cdot)$, the closed unit ball around $X^c$ (a drift-controlled Wiener sausage \cite{donsk}).

The objective function $J$ (to be minimised) is 

$$
J(c)=\int_0^{\tau^c} e^{-\alpha t}\mu(B(X^c_t))dt+e^{-\alpha \tau^c}\kappa1_{(\tau^c<\infty)},
$$
where $\tau^c$ is a stopping time (time of retirement) which we also control and $\kappa$ is the cost of retirement.

This example is formalized and generalized as Example~\ref{search}.
\end{example}

\section{Stochastic control systems}\label{section:setting}
We begin by specifying the formal ingredients of a system of optimal dynamic stochastic control. 

\begin{definition}[Stochastic control system]\label{setting}
A \textbf{stochastic control system} consists of:
\begin{enumerate}[(i)]
\item\label{setting:T} A (linearly ordered) \bip{time set} $T$. 

{\sl We will assume (for simplicity) that either $T=\mathbb{N}_0$, or  $T=[0,\infty)$, with the usual order.} 

\item\label{setting:C} A set $\CC$ of \bip{admissible controls}. 

{\sl For example, \{equivalence classes 
of\} processes or stopping times. In general, $\CC$ is an arbitrary index set.}

\item\label{setting:Omega} A 
\bip{sample space} $\Omega$ endowed with a collection 
 of $\sigma$-algebras $(\FF^c)_{c\in \CC}$. 
 
 {\sl We regard $\FF^c$ as all the \emph{information 
 accumulated} (but not necessarily acquired by the controller) by the ``end of time'' or, possibly, by a 
 ``terminal time'', when $c$ is the chosen control. For example, in the case of optimal 
 \emph{stopping}, given a process $X$, the set of controls $\CC$ would be the \{equivalence classes of 
 the\} stopping times of the \{completed\} natural filtration of $X$, and for any $S\in \CC$, 
 $\FF^S=\sigma(X^S)$, the $\sigma$-field generated by the stopped process \{or its completion\}. }
 
\item\label{setting:ptymeasures} $(\PP^c)_{c\in \CC}$, a collection of \{complete\} probability 
measures, each $\PP^c$ having a domain which includes the \{$\PP^c$-complete\} $\sigma$-field 
$\FF^c$ (for $c\in \CC$). 

{\sl The controller chooses a probability measure from the collection 
$(\PP^c)_{c\in \CC}$. This allows for  the Girsanov approach to control, where the controller is seen as 
affecting the probability measure, rather than just the random payoff.}

\item\label{setting:payoff} A \bip{reward function} $J:\CC\to [-\infty,\infty]^\Omega$, each $J(c)$ being $\mathcal{F}^c$-measurable \{and defined up to $\PP^c$-a.s. equality\} (as $c$ runs over $\CC$). 

{\sl We further insist $\EEc J(c)^-<\infty$ for all $c\in \CC$.  Given the control $c\in \CC$, $J(c)$ is the random \emph{payoff}. Hence, in general, we allow both the payoff, as well as the probability law, to vary.}

\item\label{setting:filtrations} A collection of filtrations, indexed by $T$, $(\GG^c)_{c\in \CC}$ on 
$\Omega$. 

{\sl It is assumed $\GG_\infty^c\eqdef \lor_{t\in T}\GG^c_t\subset \FF^c$, and (for 
simplicity) that $\GG^c_0$ is $\PP^c$-trivial (for all $c\in \CC$) \{and contains all the $\PP^c$-null 
sets\}, while $\GG^c_0=\GG^d_0$ \{ so the null sets for $\PP^c$ and $\PP^d$ are the same\} and 
$\PP^c\vert_{\GG^c_0}=\PP^d\vert_{\GG^d_0}$ for all $\{c,d\}\subset \CC$. $\mathcal{G}_t^c$ is 
the \emph{information acquired} by the controller by time $t\in T$, if the control chosen is $c\in \CC$ 
(e.g. $\GG^c$ may be the \{completed\} natural filtration of an observable process $X^c$ which 
depends on $c$). Perfect recall is thus assumed.}
\end{enumerate}
\end{definition}
\noindent{\textbf{Example \ref{bms} continued.}} \label{ex-BMs-contd}
{\sl We formalize Example~\ref{bms} in the context of Definition~\ref{setting}. The time set is $[0,\infty)$. Fix the discount factor $\alpha\in (0,\infty)$, let $(\Omega,\HH,\PP)$ be a probability space supporting two independent, 
sample-path-continuous, Brownian motions $B^0=(B^0_t)_{t\in [0,\infty)}$ and $B^1=(B^1_t)_{t\in [0,\infty)}$, starting at $0$ and $-x\in \mathbb{R}$, respectively (Definition~\ref{setting}\ref{setting:Omega} and~\ref{setting}\ref{setting:ptymeasures}; $\FF^c=\HH$, $\PP^c=\PP$ for all $c$). Denote by $\FF$ the natural filtration of the pair $(B^0,B^1)$. Then, for each c\`adl\`ag, $\FF$-adapted, $\{0,1\}$-valued process $c$, let $\GG^c$ be the natural filtration of $B^c$, the observed process (Definition~\ref{setting}\ref{setting:filtrations}); let $(J_k^c)_{k=0}^\infty$ be the jump times of $c$ (with $J_0^c\eqdef 0$; $J^c_k=\infty$, if $c$ has less than $k$ jumps); and define:
\begin{equation}\label{eq:epsilon-sep}
\CC\eqdef \bigcup_{\epsilon>0}\big\{\FF\text{-adapted, c\`adl\`ag, }\{0,1\}\text{-valued, processes }c,\text{ with }c_0=0,
\end{equation}
$$\;\;\;\;\;\;\;\;\;\;\;\;\;\;\;\;\;\;\;\;\;\;\;\;\;\;\text{ that are }\GG^c\text{-predictable and such that }J^c_{k+1}-J^c_k\geq \epsilon\text{ on }\{J_k^c<\infty\}\text{ for all }k\in \mathbb{N}_0\big\}$$
(Definition~\ref{setting}\ref{setting:C}). 
A more ``obvious'' choice would be to omit the condition on \lq $\epsilon$-separation of jumps' of the control process in the definition of $\CC$; however, a moment's thought will show that an optimal control is potentially not c\`adl\`ag, since, on jumping, we are exposed to unpleasant surprises: the unobserved BM may be so highly negative that we wish to switch straight back on first observing it. The insistence on the ``$\epsilon$-separation'' of the jumps of the controls allows us to emphasize the salient features of the control-dependent informational flow, without being preoccupied by the technical details.

Next, define, for each $c\in \CC$: 
\begin{itemize}
\item  the last jump time of $c$ before time $t$ and the lag since then, respectively: 
\begin{itemize}
\item $\sigma^c_t\eqdef \sup\{s\in [0,t]:c_s\ne c_t\}$; 
\item $\tau^c_t\eqdef t-\sigma^c_t$;  
\end{itemize}
 for each $t\in [0,\infty)$ (with the convention $\sup\emptyset\eqdef 0$);
\item $Z^c\eqdef B^c-B^{1-c}_{\sigma^c}$, the current distance of the observed Brownian motion to the last recorded value of the unobserved Brownian motion;
\item $J(c)\eqdef \int_0^\infty e^{-\alpha t}Z^c_tdt-\int_{(0,\infty)} e^{-\alpha t}K(Z^c_{t-},\tau^c_{t-})\vert dc\vert_t$, where $K:\mathbb{R}\times [0,\infty)\to \mathbb{R}$ is a measurable function of polynomial growth (Definition~\ref{setting}\ref{setting:payoff}; note that $\vert Z^c\vert \leq \overline{B^0}+\overline{B^1}$ (where a line over a process denotes its running supremum), so there are no integrability issues (thanks to the `$\epsilon$-separation' of the jumps of $c$). 
\end{itemize}
Notice that $\EE^\PP  \left[\int_0^\infty e^{-\alpha t}Z^c_tdt\right]=\EE^\PP \left[\int_0^\infty e^{-\alpha t}\left(B^c_t-B^{1-c}_t\right)dt\right]$. Define $V(x)\eqdef \sup_{c\in \CC}\EE^\PP J(c)$. }
\finish

\begin{definition}[Optimal expected payoff]
We define the \emph{optimal expected payoff} 
$$v\eqdef \sup_{c\in \CC}\EEc J(c)$$  (with $\sup\emptyset\eqdef -\infty$). Then, $c\in \CC$ is said to be \emph{optimal} if $\EE^{\PP^c}J(c)=v$, while a $\CC$-valued net is said to be \emph{optimizing} if the limit of its payoffs is $v$. 
\end{definition}

\begin{remark}
\leavevmode
\begin{enumerate}
\item It is no restriction to assume the integrability of the negative part of $J$ in 
Definition~\ref{setting}\ref{setting:payoff} since, allowing controls $c$ for which $\EEc J(c)^-=\infty$, but for which $\EEc J(c)$ is defined, would not change the value of $v$. 
\item A consideration of filtering problems shows that it is \emph{not} natural to insist on each $J(c)$ 
being $\GG^c_\infty$-measurable. The outcome of our controlled experiment need never be known to 
the controller, 
all we are concerned with is the maximization of its expectation.
\item In the case where $\CC$ is a collection of processes, a natural requirement is for each such process $c\in \CC$ to be adapted or  previsible with respect to $\GG^c$. If $\CC$ is a collection of random times, then each such $c\in \CC$ should presumably be a (possibly predictable) stopping time of $\GG^c$. But we do not insist on this.
\end{enumerate}
\end{remark}
We now introduce the concept of a control time, a natural generalization of the notion of a stopping time to the setting of control-dependent filtrations.

\begin{definition}[Control times]\label{def:stoppingtimes}
A collection of random times $\SS=(\SS^c)_{c\in \CC}$ is called a \textbf{control time}, if $\SS^c$ is a \{defined up to $\PP^c$-a.s. equality\} stopping time of $\GG^c$ for every $c\in \CC$. 
\end{definition}

\begin{example}
A typical situation to have in mind is the following. We observe a process $X^c$, its values being dependent on $c$. Then define $\GG^c$ to be the \{completed\} natural filtration of $X^c$. Letting $\SS^c$ be the first entrance time of $X^c$ into some fixed set, the collection $(\SS^c)_{c\in \CC}$ constitutes a control time (as long as one can formally establish the stopping time property). \finish
\end{example}

\begin{definition}[Deterministic and control-constant times]
If $\SS^c(\omega)=a\in T\cup\{\infty\}$ for \{$\PP^c$-almost\} all $\omega\in \Omega$, and every 
$c\in \CC$, then $\SS$ is called a \textbf{deterministic} time. More generally, if there is a random time 
$S$, which is a stopping time of $\GG^c$ and $\SS^c=S$ \{$\PP^c$-a.s\} for each $c\in \CC$, then 
$\SS$ is called a \textbf{control-constant time}.
\end{definition}
As yet, $\CC$ is an entirely abstract set with no dynamic structure attached to it. The following definitions
establish this structure. The reader should think of $\mathcal{D}(c,\SS)$ as being the controls 
``agreeing \{a.s.\} with $c$ up to time $\SS$" (see Examples~\ref{ex:cts_field} and~\ref{search} for examples of the collections $\mathcal{D}(c,\SS)$).

\begin{definition}[Adaptive control dynamics]\label{setting:controldyn}
Given a  stochastic control system of Definition~\ref{setting}, the system is said to have {\bf adaptive control dynamics} if it comes equipped with a 
family $(\mathcal{D}(c,\SS))_{(c,\SS)\in \CC\times \Gtimes}$ of 
subsets of $\CC$ with the following properties: 

\begin{enumerate}[(1)]
\item\label{controldyn:zeroo} $\Gtimes$ is a collection of control times.
\item\label{controldyn:one} $c\in \DD(c,\SS)$ for all $(c,\SS)\in  \CC\times \Gtimes$.
\item\label{controldyn:two} For all $\SS\in \Gtimes$ and $\{c,d\}\subset \CC$, $d\in \DD(c,\SS)$ implies $\SS^c=\SS^d$ \{$\PP^c$  and  $\PP^d$-a.s\}.
\item\label{controldyn:five}  If $\{\SS,\TT\}\subset \Gtimes$, $c\in \CC$ and $\SS^c=\TT^c$ \{$\PP^c$-a.s\}, then $\DD(c,\SS)=\DD(c,\TT)$. 
\item\label{controldyn:three} If $\{\SS,\TT\}\subset \Gtimes$ and $c\in \CC$ for which $\SS^d\leq \TT^d$ \{$\PP^d$-a.s.\} for $d\in \DD(c,\TT)$, then $\DD(c,\TT)\subset \DD(c,\SS)$. 
\item\label{controldyn:four} For each $\SS\in \Gtimes$, $\{\mathcal{D}(c,\SS):c\in \CC\}$ is a partition of $\CC$.
\item\label{controldyn:zero} For all $(c,\SS)\in  \CC\times \Gtimes$, if $\SS^c$ is
 identically \{or $\PP^c$-a.s.\} equal to $\infty$ (respectively $0$) then $\DD(c,\SS)=\{c\}$ (respectively $\DD(c,\SS)=\CC$).
\end{enumerate}
\end{definition}

\begin{remark}
In connection with Condition~\ref{controldyn:zero} of  Definition~\ref{setting:controldyn}  see Remark~\ref{remark:extend}\ref{remark:extend:a}.
\end{remark}

\begin{definition}
When condition \ref{controldyn:one} and 
\ref{controldyn:four} of Definition~\ref{setting:controldyn} prevail, we write $\sim_\SS$ for the equivalence relation induced by the partition $\{\mathcal{D}(c,\SS):c\in \CC\}$. 
\end{definition}

\begin{remark}
When $\SS$ is not a control-constant time then condition  
 \ref{controldyn:two} of Definition~\ref{setting:controldyn}  
may bite. For example, when, for some control $c$, $\SS^c$ is the first entrance time into some fixed set of an observed controlled process $X^c$, then controls agreeing with $c$ up to $\SS^c$,  should  leave $X^c$ invariant. 
This assumption is thus as much a restriction/consistency requirement on the family $\DD$, as on which 
control times we can put into the collection $\Gtimes$. Put differently, $\Gtimes$ is not 
necessarily a  completely arbitrary collection of control times. For while a control time is just \emph{any} 
family of $\GG^c$-stopping times, as $c$ runs over the control set $\CC$, the members of $\Gtimes$ 
 enjoy the further property of ``agreeing between two controls, if the latter coincide prior to 
them''. This is of course trivially satisfied for deterministic times (and, more generally, control-constant 
stopping times), but may hold for other control times as well. 
The choice of the family $\Gtimes$ will generally be dictated by the problem at hand, and is typically to do with what times the act of ``controlling'' can be effected at. Example~\ref{counter2} below, anticipating somewhat the results of section~\ref{section:bellman},  illustrates this in the control-independent informational setting. It manifests, if one is to preserve Bellman's principle, the need to work with stopping times (and \emph{a fortiori} control times, when there is dependence on control) in general. 
\end{remark}
\begin{example}\label{counter2}
\emph{Control at stopping times.}  Fix a probability space $(\Omega,\HH,\PP)$ and on it (i) a Poisson process $N$ of unit intensity with arrival times $(S_n)_{n\in\mathbb{N}_0}$, $S_0\eqdef 0$, $S_n<\infty$ for all $n\in \mathbb{N}$; (ii) an independent sequence  of independent random signs $R=(R_n)_{n\in \mathbb{N}_0}$ with values in $\{-1,+1\}$ with  $\PP(R_n=+1)=1-\PP(R_n=-1)=2/3$. 

The ``observed process'' is 
$$W_t\eqdef N_t+\int_0^t \sum_{n\in \mathbb{N}_0}R_n\mathbbm{1}_{[S_n,S_{n+1})}(s)ds$$ 
(so we add to $N$ a drift of $R_n$ during the random time interval $[S_n,S_{n+1})$, $n\geq 0$). Let $\GG$ be the \emph{natural} filtration of $W$. Remark the arrival times of $N$ are stopping times of $\GG$. 

The set of controls $\CC$ consists of real-valued, measurable processes, starting at $0$, which are adapted to the natural filtration of the bivariate process $(W\mathbbm{1}_{\{\Delta N\ne 0\}},N)$ (where $\Delta N$ is the jump process of $N$; intuitively, we must decide on the strategy for the whole of $[S_n,S_{n+1})$ based on the information available at time $S_n$ already, $n\geq 0$). For $X\in \CC$ consider the \emph{penalty} functional 
$$J(X)\eqdef \int_{[0,\infty)} e^{-\alpha t}\mathbbm{1}_{(0,\infty)}\circ \vert X_t-W_t\vert dt,$$
where $\alpha\in (0,\infty)$. Let $v\eqdef \inf_{X\in \CC}\EE^\PP J(X)$ be the optimal expected penalty; clearly an optimal control is the process $\hat{X}$ which takes the value of $W$ at the instances which are the arrival times of $N$ and assumes a drift of $+1$ in between those instances, so that $v=1/(3\alpha)$. Next, for $X\in \CC$, let $$V^X_S\eqdef \PP\mbox{-}\essinf_{Y\in \CC,Y^S=X^S}\EE^\PP[J(Y)\vert \GG_S],\quad S\text{ a stopping time of }\GG,$$ be the Bellman system. We shall say $Y\in \CC$ is {\em conditionally admissible at time $S$ for the control $X$}, if $Y^S=X^S$ and {\em conditionally optimal at time $S$} if, in addition, $V^Y_S=\EE^\PP[J(Y)\vert \GG_S]$ $\PP$-a.s. Set $V\eqdef V^{\hat{X}}$ for short. 

We now make the following two claims. 

 \textbf{Claim 1}\label{counter2:one} The process $(V_t)_{t\in [0,\infty)}$ (the Bellman process (i.e. Bellman system at the deterministic times) for the optimal control), is not mean non-decreasing (in particular, is not a submartingale, let alone a martingale with respect to $\GG$) and admits no a.s. right-continuous version; moreover, $\EE^\PP V_t<\EE^\PP V_0$ for each $t\in (0,\infty)$.

{\rm\textbf{Proof of Claim 1.}  First, clearly $V_0=v$.  Second, for $t\in (0,\infty)$, the following control, denoted $X^\star$, is, apart from $\hat{X}$, also conditionally admissible at time $t$ for $\hat{X}$: It assumes the value of $W$ at the instances of the arrival times of $N$, and a drift of $+1$ in between those intervals, \emph{until} before (inclusive of) time $t$; strictly after time $t$ and until strictly before the first arrival time of $N$ which is $\geq t$, denoted $S_t$, it takes the values of the process which starts at the value of $W$ at the last arrival time of $N$ strictly before $t$ and a drift of $-1$ thereafter; and after (and inclusive of) the instance $S_t$, it resumes to assume the values of $W$ at the arrival times of $N$ and a drift of $+1$ in between those times. Notice also that $R_t\mathbbm{1}(t\text{ is not an arrival time of }N)\in \GG_t$, where $R_t=\sum_{n\in \mathbb{N}_0} R_n\mathbbm{1}_{[S_{n},S_{n+1})}(t)$, i.e. $R_t\mathbbm{1}(t\text{ is not an arrival time of }N)$ is the drift at time $t$, on the (almost certain) event that $t$ is not an arrival time of $N$, zero otherwise. It follows that, since $\hat{X}$ is conditionally admissible for $\hat{X}$ at time $t$: $$V_t\leq \EE^\PP[J(\hat{X})\vert \GG_t],$$ so $\EE^\PP V_t\mathbbm{1}_{\{R_t=+1\}}\leq \EE^\PP J(\hat{X})\mathbbm{1}_{\{R_t=+1\}}$; whereas since $X^\star$ is also conditionally admissible at time $t$ for $\hat{X}$: $$V_t\leq \EE^\PP[J(X^\star)\vert \GG_t],$$ so $\EE^\PP V_t\mathbbm{1}_{\{R_t=-1\}}\leq \EE^\PP J(X^\star)\mathbbm{1}_{\{R_t=-1\}}=\EE^\PP J(\hat{X})\mathbbm{1}_{\{R_t=-1\}}-\EE^\PP\int_{(t,S_t)} e^{-\alpha t}dt\mathbbm{1}_{\{R_t=-1\}}=\EE^\PP J(\hat{X})\mathbbm{1}_{\{R_t=-1\}}-\frac{1}{\alpha}e^{-\alpha t}(1-\frac{1}{1+\alpha})\frac{1}{3}$ (by the Markov property of $N$ and the independence of $R$ and $N$). Summing the two inequalities we obtain $$\EE^\PP V_t\leq v-\frac{1}{3(1+\alpha)}e^{-\alpha t},$$ implying the desired conclusion (for the nonexistence of a right continuous version, assume the converse, reach a contradiction via uniform integrability).\qed}

 \textbf{Claim 2} The process $(V_{S_n}^X)_{n\in \mathbb{N}_0}$ is, however,  a discrete-time submartingale (and martingale with $X=\hat{X}$) with respect to $(\GG_{S_n})_{n\in \mathbb{N}_0}$, for all $X\in \CC$. 

{\rm\textbf{Proof of Claim 2.} For $X=\hat{X}$, the claim follows at once from the observation that  $\hat{X}$ is conditionally optimal at each of the arrival instances of $N$. The submartingale property is shown in section~\ref{section:bellman}, on p.~\pageref{ex:control-at-stopping-times:contd}, once Bellman's principle (Theorem~\ref{Bell1}) has been established. \qed}
\end{example}

Using the dynamical structure of Definition~\ref{setting:controldyn}, a \emph{key} condition --- is as follows: 

\begin{definition}[Stability under stopping]\label{assumption:temporalconsistency}
We say that a stochastic control with adaptive dynamics system is {\bf stable under stopping} if for all $\{c,d\}\subset \CC$ and $\SS\in \Gtimes$ satisfying $c\sim_\SS d$, we have $\GG^c_{\SS^c}=\GG^d_{\SS^d}$ and $\PP^c\vert_{ \GG^c_{\SS^c}}=\PP^d\vert_{\GG^d_{\SS^d}}$.
\end{definition}
From now on, up to  and including section~\ref{section:bellman}, we shall assume that the conditions of Definitions \ref{setting}, \ref{setting:controldyn} and \ref{assumption:temporalconsistency} all hold i.e. we are dealing with a {\em stochastic control system which has adaptive dynamics and is stable under stopping} which we shall refer to simply as a {\bf coherent control system}.
\begin{remark}
So, from now on it is appropriate to think of  $\DD(c,\SS)$ as just  $\DD(c,\SS^c)$, the collection of admissible controls which: up to the stopping time $\SS^c$, agree with $c$; generate the same information and have the same control law as $c$.
\end{remark}

\noindent{\textbf{Example \ref{bms} continued.}}\label{ex-BMs-contd-1} {\sl For $c\in \CC$ and control times $\SS$, set 
$$\DD(c,\SS)\eqdef \{d\in \CC:d^{\SS^c}=c^{\SS^c}\}.
$$
Then let $\Gtimes$ be any subset of
$$
\Gtimes'\eqdef \{\text{control times }\SS\text{ such that whenever }d\in \DD(c,\SS) \text{ then }\SS^c=\SS^d\,  \text{and }\,  \GG^c_{\SS^c}=\GG^d_{\SS^d}\}.$$ This defines a coherent control system.  Clearly the deterministic times $[0,\infty)\subset \Gtimes'$. Moreover, thanks to Corollary~\ref{theorem:observational_consistency} and Remark~\ref{remark:important},  $\Gtimes'=\{\text{control times }\SS\text{ such that }\forall c\in \CC\; \forall d\in \CC\;\;(d^{\SS^c}=c^{\SS^c}\text{ implies that } \SS^c=\SS^d)\}$, provided $(\Omega,\HH)$ is Blackwell. Note the freedom in the choice of $\Gtimes$ (as long as it is a subset of $\Gtimes'$); see also Remark~\ref{fundamental}\ref{ceteris-paribus} to follow.
}\finish

\newcounter{dummy}
\refstepcounter{dummy}
\label{bullet:three}

We will conclude this section with rather general extensions of 
Examples \ref{bms} and \ref{rand} illustrating the concepts introduced thus far, focusing on the control-dependent informational flows, and with explicit references made to Definitions~\ref{setting}, ~\ref{setting:controldyn} and \ref{assumption:temporalconsistency}.

Before doing so, we summarize  the remarks and observations regarding control systems made so far in

\begin{remark}\label{fundamental}\textbf{Fundamental remark.}
\begin{enumerate}[(I)]
\item As far as the optimal value and optimal control (or optimizing net of controls) are concerned, \emph{only} the objects $(\Omega,\CC,(\FF^c)_{c\in \CC},J,(\PP^c)_{c\in \CC})$ need be specified. Indeed,  given these objects satisfying items \ref{setting:C}-\ref{setting:payoff} of Definition~\ref{setting}, one can define $T\eqdef \mathbb{N}_0$ or $T\eqdef [0,\infty)$ , $\Gtimes\eqdef \emptyset$ and $\GG^c_t\eqdef \{\emptyset,\Omega\}$ for all $c\in \CC$ and $t\in T$, and the system is a coherent control system, and also (trivially) satisfies Assumption~\ref{assumption:lattice} (to follow). Nevertheless, this is obviously a static description that ignores any dynamical structure. 
\item Of course, the $\GG^c$s (which require the presence of $T$) will typically be used to define $\CC$ by insisting that an admissible control $c$ be adapted/predictable/a stopping time with respect to $\GG^c$ (meaning, in particular, that the $\GG^c$s should first be given  \emph{a priori} on a larger class of controls $c$, not just the ones that ultimately end up being admissible). But  we do not insist on this, chiefly since it is not necessary to limit the controls to being processes or random times. 
\item  The rest of the stochastic control system structure is there to reflect a {\em dynamic}\  control setting and to facilitate its analysis.  In particular, the equivalence classes, $\DD(c,\SS)$, specify a dynamical structure for the admissible controls $c\in \CC$ and those control times $\SS$, for which this can be done both consistently (see Definitions~\ref{setting:controldyn} and ~\ref{assumption:temporalconsistency}) and informatively (see Bellman's principle, Theorem~\ref{Bell1}, below). 
\item\label{ceteris-paribus} \emph{Ceteris paribus}, the bigger $\Gtimes$, the more informative is Bellman's principle. But there are restrictions on the members of $\Gtimes$ (especially those of Definition~\ref{assumption:temporalconsistency} and Assumption~\ref{assumption:lattice} to follow), which will generally preclude some control times from being included in $\Gtimes$. Specifically, Assumption~\ref{assumption:lattice} includes an interplay between $(\Omega,(\FF^c)_{c\in \CC},(\PP^c)_{c\in \CC}, J,\CC)$ and $(T,(\GG^c)_{c\in \CC},\Gtimes, (\DD(c,\SS))_{c\in \CC,\SS\in \Gtimes})$, the parts of the control system that, respectively, determine and help analyze the stochastic problem at hand (see also Remark~\ref{remark:extend}). In particular, Assumption~\ref{assumption:lattice} can fail for deterministic times (see Example~\ref{counter2} as continued on p.~\pageref{ex:control-at-stopping-times:contd}).
\end{enumerate}
\end{remark}

Remark~\ref{fundamental} is most easily appreciated in the context of the two examples that follow. The first is a generalization of Example~\ref{bms} 
to motion in a random field.

\begin{example}\label{ex:cts_field}
\emph{Motion in a random field.} The time set is $[0,\infty)$ (Definition~\ref{setting}\ref{setting:T}; $T=[0,\infty)$). We are given: 
\begin{enumerate}
\item a filtered probability space $(\Omega,\HH,\FF,\PP)$ (Definitions~\ref{setting}\ref{setting:Omega} and ~\ref{setting}\ref{setting:ptymeasures}; $\FF^c=\HH$ and $\PP^c=\PP$ for each $c$);
\item\begin{itemize}
\item a subset $O\subset \mathbb{R}$; 
\item an ``initial point'' $o_0\in O$; 
\item a random (time-dependent) real-valued field $(Y^o)_{o\in O}$ -- each $Y^o_t$ being an $\FF_t$-measurable random variable, and the random map $((o,t)\mapsto Y^o_t)$ being assumed continuous from $O\times [0,\infty)$ into $\mathbb{R}$ (so that the map $((\omega,o,t)\mapsto Y^o_t(\omega))$ is automatically $\HH\otimes \mathcal{B}(O)\otimes \mathcal{B}([0,\infty))/\mathcal{B}(\mathbb{R})$-measurable)  [in Example~\ref{bms}, $O=\{0,1\}$, $o_0=0$ and $Y^0=B^0$ and $Y^1=B^1$ are the two Brownian motions]; 
\end{itemize}
\item ``discount factor $\alpha\in [0,\infty)$; 
\item a radius of observation $a\in [0,\infty)$ [in Example~\ref{bms}, $a=0$]; 

and 
\item a family $\HHH$  of subsets of $O^{[0,\infty)}$,  \emph{directed upwards with respect to union} i.e. if $A,C\in \HHH$ then $\exists D\in \HHH$ with $D\supset A\cup C$  [in Example~\ref{bms}, $\HHH=\{\{p\in \{0,1\}^{[0,\infty)}: p\text{ is c\`adl\`ag and for all }k\in \mathbb{N},\text{  }J(p)_{k+1}-J(p)_k\geq \epsilon\text{ if }J(p)_k<\infty\}:\epsilon\in (0,\infty)\}$, where $J(p)_k$ is the time of the $k$-th jump of the path $p$ ($=\infty$, if there is not one)].
\end{enumerate}
Denoting $B_a(0)\eqdef \{x\in \mathbb{R}:\vert x\vert\leq a\}$, we require that $(Y^o_0)_{o\in B_a(0)}$ is $\PP$-trivial and if $a>0$ that $O=\mathbb{R}$. With regard to the random field $Y$ think of, for example, the local times of a Markov process, the Brownian sheet, solutions to SPDEs \cite{brzezniak} etc. 

Let $\CC'$ consist of precisely all the 
$\FF$-adapted, finite variation right-continuous 
processes $c$ with $c_0=o_0$ such that $c$ takes values in $H$ for some $H\in \HHH$ ($H$ may vary with the choice of $c$). For each $c\in \CC'$, we assume we are also given 
\begin{enumerate}[resume]
\item an $\HH\otimes \mathcal{B}([0,\infty))$-measurable ``reward'' process $Z^c$ taking values in a measurable space $(A,\AA)$ [in Example~\ref{bms}, $(A,\AA)=(\mathbb{R},\mathcal{B}(\mathbb{R}))$]; 
\item and a $\HH\otimes \mathcal{B}([0,\infty))$-measurable ``penalty'' process $\Gamma^c$ taking values in a measurable space $(E,\EEE)$ [in Example~\ref{bms}, $\Gamma^c:=(Z^c,\tau^c)$ and $(E,\EEE)=(\mathbb{R}\times [0,\infty),\mathcal{B}(\mathbb{R}\times [0,\infty)))$].
\end{enumerate}
We assume that $\Gamma^c$ and $Z^c$ depend only ``path-by-path" on $c$ in the sense that if $c(\omega)=c'(\omega)$, then $\Gamma^c(\omega)=\Gamma^{c'}(\omega)$ and $Z^c(\omega)=Z^{c'}(\omega)$ for $\omega\in \Omega$, $\{c,c'\}\subset \CC'$.  Finally, we are given 
\begin{enumerate}[resume]
\item a measurable function $f:A\to [-\infty,\infty]$ and a measurable function $g:E\to [-\infty,\infty]$ [in Example~\ref{bms}, $f=\id_\mathbb{R}$ is the identity on $\mathbb{R}$ and $g=K$]. 
\end{enumerate}

Informally, the idea is to control the movement in such a random field via a control $c\in \CC'$, observing at time $t\in [0,\infty)$ only the values of the field  in the interval $[c_t-a,c_t+a]\cap\mathbb{R}$. What is progressively observed then, under the control $c$, is the process $R^c\eqdef ((Y^{c_t+x}_t)_{x\in B_a(0)})_{t\in [0,\infty)}$ taking values in the space $C(B_a(0),\mathbb{R})$ of continuous paths from $B_a(0)\to \mathbb{R}$ endowed with the supremum norm $\Vert\cdot\Vert_\infty$ (which makes it into a Polish space, since it is separable by the Stone-Weierstrass Theorem; and making $R^c$ right-continuous), and the corresponding Borel $\sigma$-field that coincides with the trace $\mathcal{B}(\mathbb{R})^{\otimes B_a(0)}\vert_{C(B_a(0),\mathbb{R})}$ of the product $\sigma$-field. An admissible control $c$ is either adapted or predictable (we will consider both cases -- in Example~\ref{bms} admissible controls are predictable) with respect to $\GG^c$. Furthermore, given a $c\in \CC'$: 
\begin{enumerate}[(a)] 
\item rewards accrue according to the  function $f$ of $Z^c$; 
\item the speed of movement is penalized according to the function $g$ of $\Gamma^c$; 
\item both rewards and penalties are discounted at rate $\alpha$. 
\end{enumerate}

We now formalize the preceding paragraph. The set $\CC$ of Definition~\ref{setting}\ref{setting:C} is specified as containing  precisely all $c\in \CC'$  that are adapted $\langle$predictable$\rangle$ with respect to the natural filtration (denoted $\GG^c$; as in Definition~\ref{setting}\ref{setting:filtrations}) of the process $R^c$, and satisfying (see the definition of $J$ in the paragraph following) $\EE\left[\int_0^\infty e^{-\alpha t}f^-( Z^c_t)dt+\int_0^\infty e^{-\alpha t}g^+(\Gamma^c_{t-})\vert dc\vert_t\right] <\infty$ (where $\vert dc\vert$ is the total variation of $dc$ (strictly speaking $dc$ is only defined on $\cup_{T\in [0,\infty)}\mathcal{B}([0,T])$ and is a finite, signed measure locally, but the family $((\vert dc\vert)\vert_{\mathcal{B}([0,T])})_{T\in [0,\infty]}$ admits a unique extension $\vert dc\vert$ to $\mathcal{B}([0,\infty))$ as a $\sigma$-finite measure); $\Gamma^c_{0-}\eqdef\Gamma^c_0$). Clearly the observed information $\GG^c$ depends in a highly non-trivial way on the chosen control $c$.  Remark that, for each $c\in \CC$, $\GG^c$ is a sub-filtration of $\FF$. 

Next, the payoff functional $J$ (Definition~\ref{setting}\ref{setting:payoff}) is given by: 
$$J(c)\eqdef \int_0^\infty e^{-\alpha t}f(Z^c_t)dt-\int_0^\infty e^{-\alpha t}g(\Gamma^c_{t-})\vert dc\vert_t,\quad c\in \CC.$$

Finally, with regard to Definition~\ref{setting:controldyn}, define for any $c\in \CC$ and control time $\SS$, $$\DD(c,\SS)\eqdef \{d\in \CC:d^{\SS^c}=c^{\SS^c}\},$$ and then let 
$\Gtimes$ be any subset of
$$ \Gtimes'\eqdef \{\text{control times }\SS\text{ such that }\forall c\;\forall d\;(d\in \DD(c,\SS)\text{ implies that } \SS^c=\SS^d\, \text{ and }\,  \GG^c_{\SS^c}=\GG^d_{\SS^d})\};
$$
clearly the deterministic times $[0,\infty)\subset \Gtimes'$. 
We will see that $\Gtimes'=\{$control times $\SS$ such that $\forall c\in \CC \;\forall d\in\CC\; (d^{\SS^c}=c^{\SS^c}\text{ implies that } \SS^c=\SS^d)\}$, as long as $(\Omega,\HH)$ is Blackwell (which can typically be taken to be the case).  Indeed, note that each $R^c$ is a Polish-space-valued right-continuous (in particular, progressively measurable) process, so that Corollary~\ref{theorem:observational_consistency} together with parts \ref{remark:important:2} and \ref{remark:important:3} of Remark~\ref{remark:important} apply. Regardless of whether or not $(\Omega,\HH)$ is Blackwell, however, all the provisions of Definitions~\ref{setting}, ~\ref{setting:controldyn}, and ~\ref{assumption:temporalconsistency}, are met, and so the example is a coherent control system. 
\finish 
\end{example}

We now give a generalization of the Poisson search model, Example \ref{rand}.

\begin{example}\label{search}
\emph{Random measure search model.} 
Again the time set is $[0,\infty)$ (Definition~\ref{setting}\ref{setting:T}; $T=[0,\infty)$).  We are given: 
\begin{enumerate}
\item $(\Omega,\HH, \FF,\PP)$, a filtered probability space;
\item $\mu$, a locally finite $\FF_0$-measurable random measure on $(\mathbb{R}^n,\mathcal{B}(\mathbb{R}^n))$, with a locally finite number of atoms; 
\item $\kappa\in [0,\infty]$, the cost of retirement;
\item a discount factor $\alpha\in [0,\infty)$;
\item $a\in [0,\infty)$, the observation radius [in Example \ref{rand}, $a=1$]. 
\end{enumerate}
We require that the location and the sizes of the atoms of  $\mu(B_a(0))$ are $\PP$-trivial (for Example \ref{rand}, condition on them first and fix knowledge of them).

We model the time of retirement by sending the control to a coffin state after retirement: let $\partial\notin \mathbb{R}^n$ be a coffin state, topologise $\mathbb{R}^n\cup\{\partial\}$ so that $\partial$ is an isolated point, clearly it is  a Polish space. Let $D(\mathbb{R}_+;\mathbb{R}^n\cup \{\partial\})$ be the space of c\`adl\`ag maps $m:[0,\infty)\to \mathbb{R}^n\cup \{\partial\}$, constant at $\partial$ after first hitting $\partial$,  
and such that $\mathbbm{1}_\mathbb{R}\circ m$ is locally bounded. 

Let $\CC'$ be the collection of $\FF$-adapted  random elements $c$ with values in $D(\mathbb{R}_+;\mathbb{R}^n\cup \{\partial\})$ satisfying $c_0=0$.

Of course we take $\FF^c\eqdef \HH$ and $\PP^c\eqdef \PP$ for each $c\in \CC'$ (Definitions~\ref{setting}\ref{setting:Omega} and ~\ref{setting}\ref{setting:ptymeasures}). 

To allow for extra conditions on controls (of a pathwise nature) let $\HHH$ be a family of subsets of $D(\mathbb{R}_+;\mathbb{R}^n\cup \{\partial\})$, directed upwards with respect to union [in Example \ref{rand}, take 
$$
\HHH=\{\{p\in D(\mathbb{R}_+;\mathbb{R}^n\cup \{\partial\})\text{ such that } \sup_{t< \tau(p)}||p(t)||_{\R^n}\leq 1\}\},
$$
where $\tau(p)$ is the first hitting time of $\partial$ by the path $p$. To allow instead controls that, for example, are bounded up to explosion (uniformly in $\omega$) by some constant that depends on the control, we would take $\HHH=\{\{p\in D(\mathbb{R}_+;\mathbb{R}^n\cup \{\partial\})$ such that $\sup_{t< \tau(p)}||p(t)||_{\R^n}\leq k\}:\; k\in (0,\infty)\}$.]

The corresponding ``controlled process'' is $X^c$ given by
$$
X^c_t=\int_0^tc_sds+\sigma W_t,
$$
where 
\begin{enumerate}[resume]
\item $W$ is an $\FF$-Brownian Motion and $\sigma\in\mathbb{R}$.
\end{enumerate}

Now for what we observe. We observe the BM, $W$, and the location and sizes of the atoms of $\mu$ in a ball of radius $a$ around $X^c$. Note that we don't observe the whole measure in this ball. In the Poisson special case, this is complete observation. 

To define a space in which the observations will live, for ${d\in \mathbb{N}_0}$ let $\S^d$ be the quotient under permutation of $(\mathbb{R}^n\times [0,\infty))^d$ (setting $(\mathbb{R}^n\times [0,\infty))^0\eqdef \{\emptyset\}$) endowed with the quotiented Euclidean metric.
Now define $\Theta=\R^n\times\cup_{d\in \mathbb{N}_0}\{d\}\times S^d$, where the metric on the disjoint union $\cup_{d\in \mathbb{N}_0}\{d\}\times S^d$ is defined by taking the minimum of the metric on each fiber ($\{d\}\times S^d$) with $1$, and making the distance between members of different fibers $1$. It is easy to check that  $\Theta$  is a Polish space.   

We think of $d$ as telling us how many atoms we are currently seeing around a location, and then $S^d$ gives the locations and masses.  The factor $\R^n$ corresponds to observing $W$. 

Now endow $\Theta$ with the corresponding Borel measurable structure, $\TT$, which coincides with the product measurable structure of $\mathbb{R}^n$ and of the disjoint union $\cup_{d\in \mathbb{N}_0}\{d\}\times S^d$. 

So what we observe, given $c\in \CC'$ is the $(\Theta,\TT)$-valued process $O^c\eqdef (W,(D^c,M^c))$, where at time $t\in [0,\infty)$, $D_t^c$ is the number of atoms of $\mu$ in $B_a(X^c_t)$, and then $M^c_t$ is the location and sizes of the atoms of $\mu$ in $B_a(X^c_t)$.  Note that $O^c_t$ is $\FF_t$-measurable for $t\in [0,\infty)$ \cite{olav}[p. 19, Lemma~2.1]. We set $\GG^c$ to be the natural filtration of $O^c$ (Definition~\ref{setting}\ref{setting:filtrations})). Then $\CC$ of  Definition~\ref{setting}\ref{setting:C}  is the set of those $c\in \CC'$ that are 
\begin{enumerate}[(i)]
\item predictable with respect to $\GG^c$  

and such that 
\item there is an $H\in \HHH$ (that may depend on $c$) with $c$ taking values in $H$ 

and 
\item $\EE^\PP \int_0^{\tau^c}e^{-\alpha t}\mu(B_a(X^c_t))dt+\kappa e^{-\alpha\tau^c}\mathbbm{1}(\tau^c<\infty)<\infty$, where $\tau^c$ is time of retirement of $c$ (i.e. the first hitting time of $\partial$ by $c$).  
\end{enumerate} 
For $c\in \CC$,  we set
$$
J(c)\eqdef -\int_0^{\tau^c}e^{-\alpha t}\mu(B_a(X^c_t))dt-\kappa e^{-\alpha \tau^c}\mathbbm{1}(\tau^c<\infty)
$$ 
(Definition~\ref{setting}\ref{setting:payoff}). 
Finally, we  define $\DD(c,\SS)$, $\Gtimes$ and $\Gtimes'$ exactly as in Example \ref{ex:cts_field}.  
Since   the process $O^c$ is right-continuous 
and Polish-space-valued, the same remark applies concerning  $\DD(c,\SS)$, as in  Example \ref{ex:cts_field}.
\finish
\end{example}

\section{The conditional payoff and the Bellman system}\label{section:cond_payoff_sytem}

\begin{definition}[Conditional payoff  and  Bellman system]
For $c\in\CC$ and $\SS\in \Gtimes$, we define: $$J(c,\SS)\eqdef \EE^{\PP^c}[J(c)\vert \GG^c_{\SS^c}],\text{ and then }V(c,\SS)\eqdef \PP^c\vert_{\GG^c_{\SS^c}}\mbox{-}\esssup_{d\in \DD(c,\SS)}J(d,\SS);$$ and say $c\in \CC$ is \textbf{conditionally optimal} at $\SS\in \Gtimes$, if $V(c,\SS)=J(c,\SS)$ $\PP^c$-a.s. The collection $(J(c,\SS))_{(c,\SS)\in \CC\times \Gtimes}$ is called the \textbf{conditional payoff system} and $(V(c,\SS))_{(c,\SS)\in \CC\times \Gtimes}$ the \textbf{Bellman system}. 
\end{definition}

\begin{remark}
\leavevmode
\begin{enumerate}[(i)]
\item Thanks to 
 Definition~\ref{assumption:temporalconsistency}, the essential suprema appearing in the definition of the conditional payoff system are well-defined (a.s.).
\item Also, thanks to Condition~\ref{setting:controldyn}\ref{controldyn:five}, $V(c,\SS)$ only 
depends on $\SS$ through $\SS^c$, in the sense that $V(c,\SS)=V(c,\TT)$ as soon as 
$\SS^c=\TT^c$ \{$\PP^c$-a.s.\}. Clearly the same holds true (trivially) of the system $J$.
\end{enumerate}
\end{remark}
Some further properties of the systems $V$ and $J$ follow. First,

\begin{proposition}\label{proposition:V-system}
$V(c,\SS)$ is $\GG^c_{\SS^c}$-measurable and its negative part is $\PP^c$-integrable for each $(c,\SS)\in\CC\times\Gtimes$. Moreover if $c\sim_\SS d$, then $V(c,\SS)=V(d,\SS)$ $\PP^c$-a.s. and $\PP^d$-a.s.
\end{proposition}
\begin{proof}
Measurability of $V(c,\SS)$ follows from its definition. Moreover, since each $\DD(c,\SS)$ is non-empty, the integrability condition on the negative parts of $V$ is also immediate. The last claim follows from the fact that $\DD(c,\SS)=\DD(d,\SS)$ (adaptive dynamics) and $\PP^c\vert_{\GG^c_{\SS^c}}=\PP^d\vert_{\GG^d_{\SS^d}}$ (stability under stopping), when $c\sim_\SS d$. 
\end{proof}
Second, Proposition~\ref{proposition:cond_payoff_system}, will 
\begin{enumerate}[(i)] 
\item establish that $(J(c,\SS))_{(c,\SS)\in \CC\times\Gtimes}$ 
is a $(\CC,\Gtimes)$ payoff system in the sense of Definition \ref{defncg} below, 

and 
\item give sufficient conditions for the  equality $J(c,\SS)=J(d,\SS)$ to hold $\PP^c$ and $\PP^d$-a.s. on an event $A\in \GG^c_{\SS^c}$, when $c\sim_\SS d$ (addressing the situation when the two controls $c$ and $d$ agree ``for all times'' on $A$). Some auxiliary concepts are needed for this; they are given in Definition \ref{accinf}. 
\end{enumerate}
\begin{definition}[Payoff system]\label{defncg}
A collection $X=(X(c,\TT))_{(c,\TT)\in \CC\times\Gtimes}$ of functions from $[-\infty,\infty]^\Omega$ is a \textbf{payoff system} with respect to $(\CC,\Gtimes)$, if 
\begin{itemize}
\item[(i)] $X(c,\TT)$ is $\GG^c_{\TT^c}$-measurable for all $(c,\TT)\in \CC\times\Gtimes$
\item[and] 
\item[(ii)] $X(c,\SS)=X(c,\TT)$ $\PP^c$-a.s. on the event $\{\SS^c=\TT^c\}$, for all $c\in \CC$ and $\{\SS,\TT\}\subset \Gtimes$. 
\end{itemize}
\end{definition}

\begin{definition}[Times accessing infinity]\label{accinf}
\leavevmode

\begin{enumerate}
\item  For a sequence $(t_n)_{n\in \mathbb{N}}$ of elements of $[0,\infty]$, we say it \textbf{accesses
infinity}, if $\sup_{n\in \mathbb{N}}t_n=\infty$. 
 \item If   $\PP$ is a probability measure on a sample space $\Psi$, $A\subset\Psi$, $S_n:\Psi\to [0,\infty]$ for $n\in \mathbb{N}$, and $(S_n(\omega))_{n\in \mathbb{N}}$ accesses infinity for (respectively $\PP$-almost) every $\omega\in A$, then we say $(S_n)_ {n\in \mathbb{N}}$  \textbf{accesses infinity pointwise (respectively $\PP$-a.s.) on $A$}.
\suspend{enumerate}

\noindent In the context of a stochastic control system:
\resume{enumerate}
\item If $(\SS_n)_{n\in\mathbb{N}}$ is a sequence in $\Gtimes$,  $A\subset\Omega$, $c\in \CC$ and  $(\SS^c_n(\omega))_{n\in \mathbb{N}}$ accesses infinity for \{$\PP^c$-almost\} every $\omega\in A$ then we say that \textbf{$(\SS_n)$ accesses infinity \{a.s.\} on $A$ for the control $c$}.
\end{enumerate}

\end{definition}

\begin{proposition}\label{proposition:cond_payoff_system}
$(J(c,\SS))_{(c,\SS)\in \CC\times\Gtimes}$ is a $(\CC,\Gtimes)$ payoff system. 

Moreover, given $c, d\in C$, $\S\in\G$ and $A\in \GG^c_{\SS^c}$ with  
 $c\sim_\SS d$,
 if
\begin{enumerate}[(i)]
\item\label{payoff:one} there exists a sequence $(\SS_n)_{n\in\mathbb{N}}$ from $\GG$  which is \{a.s.\} nondecreasing and accesses infinity \{a.s.\} on $A$ for both of the controls $c$ and $d$, and for which $c\sim_{\SS_n}d$ and $A\in \GG^{c}_{\SS^c_n}$ for each $n\in \mathbb{N}$;
\item[and]
\item\label{payoff:two} $ \EE^{\PP^c}[J(c)\vert \GG^c_\infty]=\EE^{\PP^d}[J(d)\vert \GG^d_\infty]$ $\PP^c$-a.s. and $\PP^d$-a.s. on $A$;
\end{enumerate} 
then $J(c,\SS)=J(d,\SS)$ $\PP^c$-a.s. and $\PP^d$-a.s. on $A$.
\end{proposition}
\begin{proof}
By definition, $J(c,\TT)$ is $\GG^c_{\TT^c}$-measurable. Taking $c\in \CC$ and
 setting $F\defto\{\TT^c=\SS^c\}$, $F\in \GG^c_{\SS^c}\cap \GG^c_{\TT^c}$ 
 and so $J(c,\TT)=J(c,\SS)$ $\PP^c$-a.s. on $F$ (applying 
 Lemma~\ref{conditioninglemma}), establishing that  $(J(c,\SS))_{(c,\SS)\in \CC\times\Gtimes}$ is a $(\CC,\Gtimes)$ payoff system.

To show that $J(c,\SS)=J(d,\SS)$ $\PP^c$-a.s. (and then, by symmetry, that $\PP^d$-a.s.) on $A$ under conditions \ref{payoff:one} and \ref{payoff:two}, we need only establish that 
\begin{equation}\label{agree}
\mathbbm{1}_A\EE^{\PP^c}[J(c)\vert \GG^c_{\SS^c}]=\mathbbm{1}_A\EE^{\PP^d}[J(d)\vert \GG^d_{\SS^d}]\;\PP^c\text{-a.s}.
\end{equation} 
Now \ref{payoff:one} and Lemma~\ref{lemma:accessinginfinity}  tell us that ${\PP^d}\vert_{\GG^d_\infty}$ and ${\PP^c}\vert_{\GG^c_\infty}$ agree on $A$, so that, taking an arbitrary $B\in \GG^c_{\SS^c}=\GG^d_{\SS^d}$, 
\begin{eqnarray*}
\EE^{\PP^c}[\EE^{\PP^d}[J(d)\vert \GG^d_\infty]\mathbbm{1}_A\mathbbm{1}_B]&=&\EE^{\PP^d}[\EE^{\PP^d}[J(d)\vert \GG^d_\infty]\mathbbm{1}_A\mathbbm{1}_B]\\
\text{hence }\EE^{\PP^c}[\EE^{\PP^c}[J(c)\vert \GG^c_\infty]\mathbbm{1}_A\mathbbm{1}_B]&=&\EE^{\PP^d}[\EE^{\PP^d}[J(d)\vert \GG^d_\infty]\mathbbm{1}_A\mathbbm{1}_B]
\text{ (by assumption \ref{payoff:two}),}\\
\text{so }\EE^{\PP^c}[J(c)\mathbbm{1}_A\mathbbm{1}_B]&=&\EE^{\PP^d}[\EE^{\PP^d}[J(d)\mathbbm{1}_A\vert \GG^d_{\SS^d}]\mathbbm{1}_B]\\
&&\text{ (by standard properties of conditional expectation})\\
\text{and thus } \EE^{\PP^c}[J(c)\mathbbm{1}_A\mathbbm{1}_B]&=&\EE^{\PP^c}[\EE^{\PP^d}[J(d)\mathbbm{1}_A\vert \GG^d_{\SS^d}]\mathbbm{1}_B]\\
&&(\text{since }\PP^c\vert_{ \GG^c_{\SS^c}}=\PP^d\vert_{\GG^d_{\SS^d}}\text{ by stability under stopping}).
\end{eqnarray*}
Since $B$ is arbitrary and $A\in \GG^c_{\SS^c}$, we conclude that (\ref{agree}) holds.
\end{proof}

\section{Bellman's principle}\label{section:bellman}

\begin{definition}[(super, sub-)martingale systems]
A collection $X=(X(c,\SS))_{(c,\SS)\in (\CC,\Gtimes)}$ of functions from $[-\infty,\infty]^\Omega$ is a $(\CC,\Gtimes)$-\textbf{martingale} (respectively \textbf{supermartingale},  \textbf{submartingale}) \textbf{system}, if for each $(c,\SS)\in \CC\times \Gtimes$ 
\begin{itemize}
\item[(i)]$X(c,\SS)$ is $\GG^c_{\SS^c}$-measurable, 
\item[(ii)]$X(c,\SS)=X(d,\SS)$ $\PP^c$-a.s. and $\PP^d$-a.s., whenever $c\sim_\SS d$, 
\item[(iii)] $X(c,\SS)$ is integrable (respectively the negative, positive part of $X(c,\SS)$ is integrable). 
\item[and]
\item[(iv)]for all $\{\SS,\TT\}\subset \Gtimes$ and $c\in \CC$ with $\SS^d\leq\TT^d$ \{$\PP^d$-a.s.\} for $d\in \mathcal{D}(c,\TT)$, 
$$\EEc[X(c,\TT)\vert\GG^c_{\SS^c}] = X(c,\SS),\text{ (respectively }\EEc[X(c,\TT)\vert\GG^c_{\SS^c}]\leq,\;\geq X(c,\SS)) \quad \PP^c\text{-a.s.}$$ 

\end{itemize}
\end{definition}
In order to be able to conclude the supermartingale property of the Bellman system (Bellman's principle), we shall need to make a further assumption.  

\begin{definition}[ULP]
Fix $\epsilon\in [0,\infty)$, $M\in (0,\infty]$, $c\in \CC$ and $\SS\in \Gtimes$. We say that $(J(d,\SS))_{d\in \DD(c,\SS)}$ has \textbf{the $(\epsilon,M)$-upwards-lattice property} if, whenever $\{d,d'\}\subset \DD(c,\SS)$, there exists a $d''\in \DD(c,\SS)$ such that 
$$J(d'',\SS)\geq (M\land J(d,\SS))\lor (M\land J(d',\SS))-\epsilon\quad\PP^c\text{-a.s.}$$
\end{definition}
\begin{assumption}[Weak upwards lattice property]\label{assumption:lattice}
For \textbf{every} $c\in \CC$, $\SS\in \Gtimes$ and $\{\epsilon,M\}\subset (0,\infty)$,  $(J(d,\SS))_{d\in \DD(c,\SS)}$ has {the $(\epsilon,M)$-upwards-lattice property}.
\end{assumption}
We shall make it explicit in the sequel when this assumption is in effect.

The following theorem gives some conditions which guarantee Assumption~\ref{assumption:lattice} holds.
\begin{theorem}\label{prop:lattice}
Let $c\in \CC$, $\SS\in \Gtimes$ and $\epsilon\in [0,\infty)$, $M\in (0,\infty]$. Then  Condition \ref{lattice:one}$\Rightarrow$\ref{lattice:two}$\Rightarrow $\ref{lattice:three}, where
\begin{enumerate}[(C1)]
\item\label{lattice:one} (i) For all $d\in \DD(c,\SS)$, $\PP^d=\PP^c$ 
\item[and]

(ii) For all $\{d,d'\}\subset \DD(c,\SS)$ and $G\in \mathcal{G}^c_{\SS^c}$, there is a $d''\in \DD(c,\SS)$ such that $J(d'')\geq M\land [\mathbbm{1}_GJ(d)+\mathbbm{1}_{\Omega\backslash G}J(d')]-\epsilon$ $\PP^c$-a.s. 
\item\label{lattice:two} For all $\{d,d'\}\subset \DD(c,\SS)$ and $G\in \mathcal{G}^c_{\SS^c}$, there is a $d''\in \DD(c,\SS)$ such that 
$$J(d'',\SS)\geq M\land [\mathbbm{1}_GJ(d,\SS)+\mathbbm{1}_{\Omega\backslash G}J(d',\SS)]-\epsilon\quad\PP^c\text{-a.s.}$$
\item\label{lattice:three} $(J(d,\SS))_{d\in \DD(c,\SS)}$ has the $(\epsilon,M)$-upwards-lattice property. 

\end{enumerate}
\end{theorem}
\begin{proof}
Implication \ref{lattice:one}$\Rightarrow$\ref{lattice:two} follows by conditioning on $\GG^c_{\SS^c}$ under $\PP^c$. Implication \ref{lattice:two}$\Rightarrow $\ref{lattice:three} follows by taking $G=\{J(d,\SS)>J(d',\SS)\}\in \GG^c_{\SS^c}$. 
\end{proof}

\begin{remark}\label{remark:extend}
\leavevmode
\begin{enumerate}[(i)]
\item\label{remark:extend:i} The upwards lattice property of Assumption~\ref{assumption:lattice} represents a direct connection between $((\PP^c)_{c\in \CC}, J)$ on the one hand and $((\GG^c)_{c\in \CC}, \Gtimes, (\DD(c,\SS))_{c\in \CC,\SS\in \Gtimes})$ on the other. It is weaker than insisting that every system $(J(c,\SS))_{c\in \CC}$ be upwards-directed (i.e. having the $(0,\infty)$-upwards lattice property), but still sufficient to allow one to conclude Bellman's (super)martingale principle (Theorem~\ref{Bell1}). 
\item\label{remark:extend:ii} A more precise understanding of the relationship between the applicability of Bellman's principle, and the linkage between  $((\PP^c)_{c\in \CC}, J)$ and $((\GG^c)_{c\in \CC},\Gtimes, (\DD(c,\SS))_{c\in \CC,\SS\in \Gtimes})$ remains open. In particular, as we have seen in Example~\ref{counter2}, it appears that whether or not a control time can feature as a member of $\Gtimes$ (whilst maintaining Bellman's principle) is related to whether or not the act of controlling can be effected at that time.
\item\label{remark:extend:a} It may be assumed without loss of generality (in the sense which follows) that $\{0,\infty\}\subset\Gtimes $. Specifically, we can always simply extend the family $\DD$, by defining $\DD(c,\infty)\eqdef \{c\}$ and $\DD(c,0)\eqdef \CC$ for each $c\in \CC$ -- neither the provisions of Section~\ref{setting} (Definitions~\ref{setting}, ~\ref{setting:controldyn} and ~\ref{assumption:temporalconsistency}) nor the validity of Assumption~\ref{assumption:lattice} being affected. 
\item Assumption~\ref{assumption:lattice} is of course trivially satisfied when the filtrations $\GG$ all consist of (probabilistically) trivial $\sigma$-fields. 
\end{enumerate}
\end{remark}

\noindent \textbf{Example~\ref{ex:cts_field} continued.}\label{ex-field-contd} {\sl We verify that in Example~\ref{ex:cts_field}, if the base $(\Omega,\FF)$ is Blackwell or the members of $\Gtimes$ are deterministic, then Property~\ref{lattice:one} from Theorem~\ref{prop:lattice} holds with $M=\infty$, $\epsilon=0$. 

Let $c_1\sim_\SS c_2$, $A\in \mathcal{G}^{c_1}_{\SS^{c_1}}=\mathcal{G}^{c_2}_{\SS^{c_2}}$. It is enough to show that $c\eqdef c_1\mathbbm{1}_A+c_2\mathbbm{1}_{\Omega\backslash A}\in \CC$ (since then we will have $J(c)=J(c_1)\mathbbm{1}_A+J(c_2)\mathbbm{1}_{\Omega\backslash A}$, thanks to the ``path-by-path'' dependence of $\Gamma^c$ and $Z^c$ on $c$; and clearly $c\sim_\SS c_1,c_2$). The control $c$ is a 
right-continuous finite variation $O$-valued process with initial value $o_0$, satisfying the requisite integrability condition on $f^-$ and $g^+$. Also, if $c_1$ takes values in $H$, $H\in \HH$ and $c_2$ takes values in $H'$, $H'\in \HH$, then since $\HH$ is upwards  directed with respect to union, there is an $H''\in\HH$ with $H''\supset H\cup H'$, and clearly $c$ takes values in $H''$. 

Noting that the filtration $\GG^c$ is included in the filtration $\FF$, it remains to check that $c$ is $\GG^{c}$-adapted $\langle$predictable$\rangle$. Setting $P\eqdef \SS^{c_1}=\SS^{c_2}$, it is sufficient to argue that  $c\mathbbm{1}_{\llbracket 0,P\rrbracket}=c_1\mathbbm{1}_{\llbracket 0,P\rrbracket}=c_2\mathbbm{1}_{\llbracket 0,P\rrbracket}$, $c\mathbbm{1}_{\llparenthesis P,\infty\rrparenthesis}\mathbbm{1}_A=c_1\mathbbm{1}_{\llparenthesis P,\infty\rrparenthesis}\mathbbm{1}_A$ and $c\mathbbm{1}_{\llparenthesis P,\infty\rrparenthesis}\mathbbm{1}_{\Omega\backslash A}=c_2\mathbbm{1}_{\llparenthesis P,\infty\rrparenthesis}\mathbbm{1}_{\Omega\backslash A}$ are all $\GG^c$-adapted $\langle$predictable$\rangle$.  

To this end, note that $P$ is a stopping time of $\GG^{c_1}$ and of $\GG^{c_2}$, and $(c_1)^{P}=c^{P} =(c_2)^{P}$, hence $(R^{c_1})^{P}=(R^c)^{P} =(R^{c_2})^{P}$. Then  $\GG^c_t\vert_{\{t\leq P\}}=\GG^{c_1}_t\vert_{\{t\leq P\}}=\GG^{c_2}_t\vert_{\{t\leq P\}}$, whilst $\GG^{c}_t\vert_{\{P<t\}\cap A}=\GG^{c_1}_t\vert_{\{P<t\}\cap A}$ and $\GG^{c}_t\vert_{\{P<t\}\cap ( \Omega\backslash A)}=\GG^{c_2}_t\vert_{\{P<t\}\cap (\Omega\backslash A)}$. Also, by  Theorem~\ref{theorem:galmarino}, Proposition~\ref{lemma:continuous}, and Proposition~\ref{proposition:information_increases} to follow (or trivially when $P$ is deterministic), all the events $\{t\leq P\}$, $(\Omega \backslash A)\cap\{P<t\}$ and $A\cap \{P<t\}$ belong to $\sigma((R^{c_1})^{P\land t})=\sigma((R^{c_2})^{P\land t})=\sigma((R^c)^{P\land t})\subset \sigma((R^{c})^t)=\GG^{c}_t$ (and, of course, $\GG^{c_1}_t\cap\GG^{c_2}_t$) for all $t\in [0,\infty)$.

Now apply Lemma~\ref{lemma:adap-predictable}.}
\finish

\noindent \textbf{Example~\ref{search} continued.}  \label{ex-search-contd} {\sl Using the same method as above, one verifies that, in  Example~\ref{search}, if the base $(\Omega,\FF)$ is Blackwell or the members of $\Gtimes$ are deterministic, then again Property~\ref{lattice:one} from Theorem~\ref{prop:lattice} holds with $M=\infty$, $\epsilon=0$. We leave the details to the reader. \finish}

\begin{theorem}\label{proposition:consistency}
[Cf. \cite[p. 94, Lemma~1.14]{elkaroui}.] Under  Assumption~\ref{assumption:lattice}, for any $c\in \CC$, $\TT\in \Gtimes$ and any sub-$\sigma$-field $\AA$ of $\GG^c_{\TT^c}$: $$\EEc[V(c,\TT)\vert\AA]=\PP^c\vert_\AA\mbox{-}\esssup_{d\in \DD(c,\TT)}\EE^{\PP^d}[J(d)\vert \AA]\quad \PP^c\text{-a.s.}$$ 
In particular, 
\begin{equation}\label{opty}
\EE^{\PP^c}V(c,\TT)=\sup_{d\in\DD(c,\TT)}\EE^{\PP^d}J(d).
\end{equation} 
\end{theorem}
\begin{proof}
By Lemma~\ref{lemma:esssup}, we have, $\PP^c$-a.s.:
\begin{eqnarray*}
\EEc[V(c,\TT)\vert\AA]\!\!\!\!&=&\!\!\!\!\PP^c\vert_{\AA}\mbox{-}\esssup_{d\in \DD(c,\TT)}\EE^{\PP^c}[\EE^{\PP^d}[J(d)\vert \GG^d_{\TT^d}]\vert\AA]\\
&=&\!\!\!\!\PP^c\vert_{\AA}\mbox{-}\esssup_{d\in \DD(c,\TT)}\EE^{\PP^d}[\EE^{\PP^d}[J(d)\vert \GG^c_{\TT^c}]\vert\AA],\text{ since }\GG^c_{\TT^c}=\GG^d_{\TT^d}\text{ and }\PP^c\vert_{ \GG^c_{ \TT^c}}=\PP^d\vert_{\GG^d_{\TT^d}},
\end{eqnarray*}
for $d\sim_\TT c$, from which the claim follows. 
\end{proof}

\begin{theorem}[Bellman's principle]\label{Bell1}
Suppose  that $\{0,\infty\}\subset \Gtimes$. Under Assumption~\ref{assumption:lattice}:
\begin{enumerate}[(B1)]
\item\label{Bell1:1}$(V(c,\SS))_{(c,\SS)\in \CC\times \Gtimes}$ is a $(\CC,\Gtimes)$-supermartingale system. 
\item\label{Bell1:2} If  $c^*\in\CC$ is optimal, then $(V(c^*,\TT))_{\TT\in\Gtimes}$ has a constant $\PP^{c^*}$-expectation (equal to the optimal value $v=\EE^{\PP^{c^*}}J(c^*)$). 
\item\label{Bell1:3}If  $c^*\in\CC$ is optimal and  $\EE^{\PP^{c^*}}J(c^*)<\infty$, then $(V(c^*,\TT))_{\TT\in\Gtimes}$ is a $\Gtimes$-martingale in the sense that 
\begin{itemize}
\item[(i)]for each $\TT\in \Gtimes$, $V(c^*,\TT)$ is $\GG^{c^*}_{\TT^{c^*}}$-measurable and $\PP^{c^*}$-integrable 
\item[and]
\item[(ii)]for any $\{\SS,\TT\}\subset \Gtimes$ with $\SS^d\leq \TT^d$ \{$\PP^d$-a.s.\} for $d\in \DD(c^*,\TT)$,  $$\EE^{\PP^{c^*}}[V(c^*,\TT)\vert \GG^{c^*}_{\SS^{c^*}}]=V(c^*,\SS)\quad\PP^{c^*}\text{-a.s.}$$
\end{itemize}
\item\label{Bell1:4}If $c^*\in\CC$ is conditionally optimal at $\SS\in\Gtimes$ and  $\EE^{\PP^{c^*}}J(c^*)<\infty$, then $c^*$ is conditionally optimal at $\TT$ for any $\TT\in \Gtimes$ satisfying $\TT^d\geq \SS^d$ \{$\PP^d$-a.s.\} for $d\in \DD(c^*,\TT)$. In particular, if $c^*$ is optimal, then it is conditionally optimal at $0$, so that if also $\EE^{\PP^{c^*}}J(c^*)<\infty$, then $c^*$ must be conditionally optimal at any $\SS\in \Gtimes$. 
\end{enumerate}
Regardless of whether or not Assumption~\ref{assumption:lattice} holds:
\begin{enumerate}[(B5)]
\item\label{Bell1:5}If $c^*\in  \CC$ and $\Gtimes$ includes a sequence $(\SS_n)_{n\in\mathbb{N}_0}$ for which 
\begin{itemize}
\item[(i)]$\SS_0=0$, 
\item[(ii)] the family $(V(c^*,\SS_n))_{n\geq 0}$ has a constant $\PP^{c^*}$-expectation and is uniformly integrable, \item[and]
\item[(iii)]$V(c^*,\SS_n)\to V(c^*,\infty)$, $\PP^{c^*}$-a.s. (or even just in $\PP^{c^*}$-probability), as $n\to\infty$,\end{itemize}
then $c^*$ is optimal.
\end{enumerate}
\end{theorem}
\begin{remark}
Recall that the assumption $\{0,\infty\}\subset \Gtimes$ is innocuous (see Remark~\ref{remark:extend}\ref{remark:extend:a}).
\end{remark}

\begin{proof}
Let $\{\SS,\TT\}\subset \Gtimes$ and $c\in \CC$ with $\SS^d\leq\TT^d$ \{$\PP^d$-a.s.\} for $d\in \mathcal{D}(c,\TT)$. Then, since $\SS^c\leq \TT^c$ \{$\PP^c$-a.s.\},  $\GG^c_{\SS^c}\subset \GG^c_{\TT^c}$, and  (by adaptive dynamics) $\DD(c,\TT)\subset \DD(c,\SS)$, so that we obtain via Theorem~\ref{proposition:consistency}, $\PP^c$-a.s.,
\begin{eqnarray*}
\EEc[V(c,\TT)\vert\GG^c_{\SS^c}]
&=&\PP^c\vert_{\GG^c_{\SS^c}}\mbox{-}\esssup_{d\in \DD(c,\TT)}\EE^{\PP^d}[J(d)\vert \GG_{\SS^c}^c]\\
&\leq &\PP^c\vert_{\GG^c_{\SS^c}}\mbox{-}\esssup_{d\in \DD(c,\SS)}\EE^{\PP^d}[J(d)\vert \GG_{\SS^c}^c]=V(c,\SS)
\end{eqnarray*}
(since $\GG^c_{\SS^c}=\GG^d_{\SS^d}$ for $s\sim_\SS d$) which,  together with Proposition~\ref{proposition:V-system},  establishes \ref{Bell1:1}. Assertion \ref{Bell1:2} follows at once from (\ref{opty}). To establish the martingale property \ref{Bell1:3} let $c^*$ be optimal and $\{\SS,\TT\}\subset \Gtimes$ with $\SS ^d\leq \TT^d$ \{$\PP^d$-a.s.\} for $d\in \DD(c^*,\TT)$. Note that by the supermartingale property, $v= \EE^{\PP^{c^*}}\EE^{\PP^{c^*}}[V(c^*,\TT)\vert \GG^{c^*}_{\SS^{c^*}}]\leq  \EE^{\PP^{c^*}}V(c^*,\SS)=v$. So, if $v<\infty$, 
we conclude that $V(c^*,\TT)$ is $\PP^{c^*}$-integrable, and the martingale property follows.

Now if $c^*$ is conditionally optimal at $\SS$, $\EE^{\PP^{c^*}}J(c^*)<\infty$, and $\SS^d\leq \TT^d$ \{$\PP^d$-a.s.\} for $d\in \DD(c^*,\TT)$, then since $V$ is a $(\CC,\Gtimes)$-supermartingale system, $\EE^{\PP^{c^*}}J(c^*)=\EE^{\PP^{c^*}}J(c^*,\SS)=\EE^{\PP^{c^*}}V(c^*,\SS)\geq \EE^{\PP^{c^*}}V(c^*,\TT)$. On the other hand, for sure, $V(c^*,\TT)\geq J(c^*,\TT)$, $\PP^{c^*}$-a.s., so $\EE^{\PP^{c^*}}V(c^*,\TT)\geq \EE^{\PP^{c^*}} J(c^*,\TT)=\EE^{\PP^{c^*}}J(c^*)$ hence we must have  $V(c^*,\TT)= J(c^*,\TT)$, $\PP^{c^*}$-a.s., i.e. $c^*$ is conditionally optimal at $\TT$. So \ref{Bell1:4} holds.

Finally, notice that, under the assumptions in \ref{Bell1:5}, $V(c^*,\SS_n)\to V(c^*,\infty)$ in $L^1(\PP^{c^*})$, as $n\to\infty$, and so 
$$v=\sup_{c\in \CC}\EE^{\PP^c}J(c)=\EE^{\PP^{c^*}} V(c^*,0)=\EE^{\PP^{c^*}} V(c^*,\SS_n)\nto \EE^{\PP^{c^*}} V(c^*,\infty)=\EE^{\PP^{c^*}} J(c^*)$$
\end{proof}

\begin{theorem}[Supermartingale envelope]\label{BellChar}
Under  Assumption~\ref{assumption:lattice}, 
$V$ is the \textbf{minimal $(\CC,\Gtimes)$-supermartingale system $W$} satisfying the terminal condition 
\begin{equation}\label{term}
W(c,\infty)\geq \EE^{\PP^c}[J(c)\vert\GG^c_\infty]\quad\PP^c\text{-a.s. for each }c\in \CC.
\end{equation}
\end{theorem}
\begin{proof}
That $V$ is a $(\CC,\Gtimes)$-supermartingale system satisfying (\ref{term}) is clear from the definition of $V$ and Theorem~\ref{Bell1}. Next, let $W$ be a $(\CC,\Gtimes)$-supermartingale system satisfying (\ref{term}). Then for all $(c,\TT)\in \CC\times\Gtimes$ and $d\in \DD(c,\TT)$, 
 $$W(c,\TT)=W(d,\TT)\geq \EE^{\PP^d}[W(d,\infty)\vert \GG^d_{\TT^d}]\geq  \EE^{\PP^d}[\EE^{\PP^d}[J(d)\vert\GG^d_\infty]\vert \GG^d_{\TT^d}]=J(d,\TT)\;\; \PP^c\text{-a.s. and  }\PP^d\text{-a.s.}$$ 
 Thus $W(c,\TT)\geq V(c,\TT)$, $\PP^c$-a.s.
\end{proof}
\begin{remark}
Notice that we recover from Theorem~\ref{BellChar} the Snell envelope characterisation of solutions to the problem of optimally stopping a c\`adl\`ag process $X$ adapted to a filtration $\FF=(\FF_t)_{t\in T}$ by taking $\CC$ to be the collection of $\FF$-stopping times, $\tau$,  and setting, for $t\in T$, $\GG^\tau_t=\FF_{t\wedge\tau}$. 
\end{remark}

\noindent \textbf{Example~\ref{counter2} continued.}\label{ex:control-at-stopping-times:contd}  {\sl In order to complete the argument in Example~\ref{counter2}, we establish the relevant lattice property. 
For arbitrary $X\in \CC$, $n\in \mathbb{N}_0$, $G\in \GG_{S_n}$, and $\{Y,Z\}\subset \CC$ with $Y^{S_n}=X^{S_n}=Z^{S_n}$, we find that the control
$U$, given by 
$$
U_t=Y_t\1_{[0,S_n]}(t)+Y_t\1_{(S_n,\infty)}(t)\1_G+Z_t\1_{(S_n,\infty)}(t)\1_{G^c},
$$
which coincides with $Y$ and $Z$ on $[0,S_n]$ and then with $Y$ on $G$ and $Z$ on 
$G^c$ strictly after $S_n$, is conditionally admissible at time $S_n$ for $X$. By Theorem~\ref{prop:lattice}\ref{lattice:one} we see that the family $\{\EE^\PP[J(Y)\vert \GG_{S_n}]:Y\in \CC, Y^{S_n}=X^{S_n}\}$ is directed downwards for each $n\in \mathbb{N}_0$. Hence, setting $S_\infty\eqdef\infty$, one can apply Theorem~\ref{Bell1} with $\Gtimes\eqdef\{S_n:n\in \mathbb{N}_0\cup \{\infty\}\}$ and with $\DD(X,S_n)\eqdef \{Y\in \CC:Y^{S_n}=X^{S_n}\}$ for $X\in \CC$ and $n\in \mathbb{N}_0\cup \{\infty\}$. 

We note that it also follows from Theorem~\ref{Bell1} and \textbf{(1)} on p.~\pageref{counter2:one}, that if $\Gtimes$ 
 contains  a strictly positive finite deterministic time $t$ (with the equivalence class $\DD(c,t)$ being defined in the obvious way), then Assumption~\ref{assumption:lattice} \emph{must fail}.}
 \finish

\section {A solved example}\label{section:formal_example}

We now conclude Example~\ref{bms}. { Recall the notation of section~\ref{section:setting} and Example \ref{bms} (as continued on pp.~\pageref{ex-BMs-contd} and~\pageref{ex-BMs-contd-1}) and note that all the conditions of section~\ref{section:setting} are satisfied. Recall also that Example \ref{bms} is a special case of Example~\ref{ex:cts_field}, which (see p.~\pageref{ex-field-contd}) satisfies the weak upwards lattice property (Assumption~\ref{assumption:lattice}), provided $(\Omega,\HH)$ is Blackwell or $\Gtimes$ contains  only deterministic times.

Now  we shall give the solution in some fairly straightforward cases.}

\begin{proposition}\label{bmsoln}
\leavevmode
\begin{enumerate}[(a)]
\item\label{example:case:a}If $K(z,t)=-2z/\alpha$, $(z,t)\in \mathbb{R}\times [0,\infty)$,  then the optimal payoff, $V$ satisfies $V(x)=x/\alpha$ and any control achieves it.
\item\label{example:case:b}If $K$ has the form:
\footnotesize $$K(z,t)=\int_\mathbb{R}\left(\frac{\vert \sqrt{t}u-z\vert-\vert z\vert}{\alpha}+\frac{e^{-\gamma\vert \sqrt{t}u-z\vert}-e^{-\gamma \vert z\vert}}{\alpha\gamma}\right)\phi(u)du+\mathbbm{1}_{(0,\infty)}(z)L(z,t),\, (z,t)\in \mathbb{R}\times [0,\infty),$$ 
\normalsize 
where  $\gamma\eqdef  \sqrt{2\alpha}$, $\phi$ is the standard normal density and $L(z,t)$ is nonnegative, measurable, and  of polynomial growth (uniformly in $t$),
then  $V(x)$ is  a symmetric function of the parameter $x$. Moreover, letting $c^\epsilon$ be the control which waits an  amount of time $\epsilon$ both after each time it jumps and at time zero, and thereafter jumps at the first entrance time of $Z^c$ into $(-\infty,0]$,  for any such $K$, $\EE^\PP J(c^\epsilon)\to V(x)=\frac{\gamma\vert x\vert+e^{-\gamma \vert x\vert}}{\alpha\gamma}$, as $\epsilon\downarrow 0$.
\end{enumerate}
\end{proposition}
\begin{remark}
We see here, in the jump times of $c^\epsilon$, an example of a whole sequence of non-control-constant control times (members of $\Gtimes'$, provided $(\Omega,\HH)$ is Blackwell).\color{black}
\end{remark}

According to Bellman's principle (Theorem~\ref{Bell1}) and the strong Markov property, for each $c\in \CC$, the following process (where $V$ is, by a slight abuse of notation, the value function\footnote{More precisely, for $z\in \mathbb{R}$, $u\in [0,\infty)$, $V(z,u)$ is the optimal payoff of the related optimal control problem in which, \emph{ceteris paribus}, $B^1=z+H_{u+\cdot}$, for a Brownian motion $H$ independent of $B^0$.}): \begin{equation}\label{Bellman_process}
\tilde S_t^c\eqdef \int_0^te^{-\alpha s}Z^c_sds-\int_{(0,t]}e^{-\alpha s}K(Z^c_{s-},\tau^c_{s-})\vert dc_s\vert+e^{-\alpha t}V(Z^c_t,\tau^c_t),
\end{equation}
 should be a $(\GG^c,\PP)$-supermartingale (in $t\in [0,\infty)$). Moreover, if an optimal strategy $c^*$ exists, then $\tilde S^{c^*}$ should be a $(\GG^{c^*},\PP)$-martingale (or,  when dealing with a sequence/net of optimizing controls, the corresponding processes should, \emph{in expectation}, `be increasingly close to being one'). 
 
 This leads us to frame the following:

\begin{lemma}[Verification Lemma]\label{verif}
\begin{enumerate}[UP]
Suppose that
\item\label{V4}the function $K(z,t)$ has uniformly (in $t$) polynomial growth in $z$,
\end{enumerate}
\begin{enumerate}[V(i)]
and that the function $h:\mathbb{R}\to\mathbb{R}$ satisfies
\item\label{V2} $h$ is of class $C^1$ and is twice differentiable, with a second derivative, which is continuous except possibly at finitely many points, where the left and right derivatives exist and are continuous from the left, respectively right; 
\item\label{V3}$h$, $h'$ and $h''$  have polynomial growth.
\end{enumerate}
Let $l\in \R$ and suppose that:
\begin{enumerate}[W(i)]
\item\label{eq:supermtg:cts} $ z-\alpha h(z)+\frac{1}{2}h''(z)\leq 0$, for a.e. $z\in \mathbb{R}$;
\item\label{eq:supermtg:jump}$-K(z,t)+\int_{\mathbb{R}}[h(\sqrt{t}u-z)-h(z)]\phi(u)du\leq 0$, for all $z\in \mathbb{R},\, t\in [0,\infty)$;
\item\label{eq:mtg:cts} $z-\alpha h(z)+\frac{1}{2}h''(z)= 0$, for a.e. $z\geq -l$;
\item\label{eq:mtg:jump} $-K(z,t)+\int_{\mathbb{R}}[h(\sqrt{t}u-z)-h(z)]\phi(u)du=0$, for all $z\leq -l,\, t\in (0,\infty)$.
\end{enumerate}
Then $V=h$ and 
an optimizing (as $\epsilon\downarrow 0$) net of optimal controls is
 $(c^\epsilon)_{\epsilon>0}$ where $c^\epsilon$ is the control which waits for a period of 
 time $\epsilon$ after each jump of $c^\epsilon$ and also at the start, and thereafter 
 switches the observed Brownian motion precisely at the first entrance time of 
 $Z^{c^\epsilon}$ into the set $(-\infty,-l]$. Remark that $c^\epsilon$ is previsible with 
 respect to $\GG^{c^\epsilon}$. 
\end{lemma}


\begin{proof}
For each control $c$, define the process $S^c$ by
$$
S^c_t\eqdef \int_0^te^{-\alpha s}Z^c_sds-\int_{(0,t]}e^{-\alpha s}K(Z^c_{s-},\tau^c_{s-})\vert dc\vert_s+e^{-\alpha t}h(Z^c_t)
$$
(as in \eqref{Bellman_process}). 

The semimartingale decomposition of $S^c$ may then be effected relative to the completed measure $\overline{\PP}$ and the usual augmentation $\overline{\GG^c}_+$ of $\GG^c$, with respect to which $Z^c$ is a semimartingale (indeed, its jump part is clearly of finite variation, whilst its continuous part is, in fact, a Brownian motion relative to the augmentation of the natural filtration of $(B^0,B^1)$). Thanks to\ref{V2} and \ref{V3} we  obtain, by the It\^o-Tanaka-Meyer formula \cite[p. 214, Theorem~IV.70, p. 216, Corollary~IV.1]{protter}, $\overline{\PP}$-a.s. for all $t\in [0,\infty)$: 
\begin{eqnarray}\label{doob-meyer}
S^c_t&=&h(x)+\underbrace{\int_0^t e^{-\alpha s}Z^c_{s-}ds}_{\eqdef C^{1}}+ \underbrace{\int_{(0,t]}e^{-\alpha s}(-K(Z_{s-}^c,\tau^c_{s-}))\vert dc\vert_s}_{\eqdef D^{1}}+\underbrace{\int_0^te^{-\alpha s} (-\alpha)h(Z^c_{s-})ds}_{\eqdef C^{2}}\\\nonumber
&&+\overbrace{\int_0^t e^{-\alpha s}h'(Z^c_{s-})dZ^c_s}^{(1)}+\underbrace{\int_0^te^{-\alpha s}\frac{1}{2}h''(Z^c_{s-})d\overbrace{[Z^c]^{\text{cts}}_s}^{s}}_{\eqdef C^{3}}+\sum_{0<s\leq t}e^{-\alpha s}\left[\overbrace{\Delta h(Z^c_s)}^{(2)}\overbrace{-h'(Z^c_{s-})\Delta Z^c_s}^{(1)}\right].
\end{eqnarray}
Note that the  parts labelled  (1) add to 
\begin{equation}\label{mtg1}
\underbrace{\overbrace{\int_0^te^{-\alpha s}h'(Z^c_{s-})d(Z^c)^{\text{cts}}_s}^{(1)}}_{\eqdef M^{1}},
\end{equation}
which is a $(\overline{\GG^c}_+,\overline{\PP})$-martingale in $t\in [0,\infty)$ (since $\vert Z^c\vert\leq \overline{B^0}+\overline{B^1}$). On the other hand, the compensator of the  term labelled (2) is 
\begin{equation}\label{compensator}
\underbrace{\overbrace{\int_{(0,t]} \vert dc\vert_s e^{-\alpha s}\left[\int_{\mathbb{R}}du\phi(u)\left(h(\sqrt{\tau^c_{s-}}u-Z^c_{s-})-h(Z^c_{s-})\right)\right]}^{(2)}}_{\eqdef D^{2}},
\end{equation} 
making 
\begin{equation}\label{mtg2}
\underbrace{\overbrace{\int_{(0,t]} \vert dc\vert_s e^{-\alpha s}\left[\Delta h(Z^c_s)-\int_{\mathbb{R}}du\phi(u)\left(h(\sqrt{\tau^c_{s-}}u-Z^c_{s-})-h(Z^c_{s-})\right)\right]}^{(2)}}_{\eqdef M^{2}}
\end{equation}
 into a $(\GG^c,\PP)$-martingale (in $t\in [0,\infty)$). For, if $\tau$ is a predictable stopping time with respect to some filtration (in continuous time) $\mathcal Z$, $U$ is a $\mathcal{Z}_\tau$-measurable random variable, and $\QQ$ a probability measure with $\QQ[\vert U\vert\mathbbm{1}_{[0,t]}\circ \tau]<\infty$ for each $t\in [0,\infty)$, then the compensator of $U\mathbbm{1}_{\llbracket \tau,\infty\rrparenthesis}$ (relative to $(\mathcal{Z},\QQ)$) is $\QQ[U\vert \mathcal{Z}_{\tau-}]\mathbbm{1}_{\llbracket \tau,\infty\rrparenthesis}$. This fact may be applied to each jump time of the $\GG^c$-predictable process $c$ (since $\vert Z^c\vert\leq \overline{B^0}+\overline{B^1}$). Then linearity allows us to conclude that $M^{2}$ is a $(\G^c,P)$-martingale (thanks to the `$\epsilon$-separation' of the jumps of $c$).

Note  that the properties of being a c\`adl\`ag (super)martingale \cite[p. 173, Lemma II.67.10]{rogers} or predictable process (of finite variation) are preserved when passing to the usual augmentation of a filtered probability space. Therefore it follows that,  relative to $(\overline{\GG^c}_+,\overline{\PP})$, $M\eqdef M^{1}+M^{2}$ is a martingale, whilst $D^{}\eqdef D^{1}+D^{2}$  (respectively  $C\eqdef C^{1}+C^{2}+C^{3}$) is a pure-jump (respectively continuous) predictable process of finite variation. 
Now conditions \ref{eq:supermtg:cts} and \ref{eq:supermtg:jump} show that $C$ and $D^{}$ 
are nonincreasing and so we have obtained the Doob-Meyer decomposition of $S^c=h(x)+M+D^{}+C$ (\cite[p. 32, Corollary~3.16]{jacod} \cite[p. 412, Theorem~22.5]{kallenberg}), and
 $S^c$ is  a $(\overline{\GG^c}_+,\overline{\PP})$-supermartingale.  
 
 Assumptions \ref{V3} and \ref{V4} ensure that 
 $S^c_t\lto J(c)\eqdef S^c_\infty$ and so we conclude from Theorem~\ref{BellChar} (applied with $W=S$ and $\Gtimes=[0,\infty]$) that
 $$h\geq V.
 $$
 
 To establish the reverse inequality, we will show that
 $\E^\PP J(c^\epsilon)\rightarrow h(x)\text{ as }\epsilon\downarrow 0$.

First, note that since $S^c\lto J(c)$ it is sufficient to check that  $C_\infty(c^\epsilon)$ and
$D^{}_\infty(c^\epsilon)$,  the limiting  values of the continuous and pure-jump components in the Doob-Meyer decomposition of $S^{c^\epsilon}$, converge to 0 in $L^1$  as $\epsilon\downarrow 0$. 
Now from 
\ref{eq:mtg:cts},
$$
C_\infty({c^\epsilon})=\int_0^\infty e^{-\alpha s}[Z^{{c^\epsilon}}_s-\alpha h(Z^{{c^\epsilon}}_s)+\half h''(Z^{{c^\epsilon}}_s)]ds$$
$$=\int_0^\infty e^{-\alpha s}[Z^{{c^\epsilon}}_s-\alpha h(Z^{{c^\epsilon}}_s)+\half h''(Z^{{c^\epsilon}}_s)]\1_{(-\infty,-l)}\circ Z^{{c^\epsilon}}_s ds
$$
and from \ref{V3} we see that to show that this  converges in $L^1$ to 0
it is sufficient, by the Dominated Convergence Theorem, to show that $\1_{(-\infty,-a)}\circ Z^{c^\epsilon}\vert_{[0,T]}\rightarrow 0$, in $\PP\times \text{Leb}\vert_{\mathcal{B}([0,T])}$-measure for each $T\in [0,\infty)$ and $a\in [l,\infty)$, which is clear from the definition of $c^\epsilon$. 

Turning to the jump part, 
$$
D_\infty({c^\epsilon})=\int_0^\infty e^{-\alpha s}
\biggl(\left[\int_{\mathbb{R}}du\phi(u)\left(h(\sqrt{\tau^{c^\epsilon}_{s-}}u-Z^{c^\epsilon}_{s-})-h(Z^{c^\epsilon}_{s-})\right)\right]-K(Z_{s-}^{c^\epsilon},\tau^{c^\epsilon}_{s-})\biggr)
\vert dc\vert_s,
$$
and so  from \ref{eq:mtg:jump} it follows that
$$
D_\infty({c^\epsilon})=\int_0^\infty e^{-\alpha s}
\biggl(\left[\int_{\mathbb{R}}du\phi(u)\left(h(\sqrt{\tau^{c^\epsilon}_{s-}}u-Z^{c^\epsilon}_{s-})-h(Z^{c^\epsilon}_{s-})\right)\right]-K(Z_{s-}^{c^\epsilon},\tau^{c^\epsilon}_{s-})\biggr)\1_{(-l,\infty)}\circ Z^{c^\epsilon}_{s-}
\vert dc\vert_s,
$$ 
and since $c^\epsilon$ never jumps when $Z^{c^\epsilon}\in (-l,\infty)$, we see that $D_\infty(c^\epsilon)=0$.
\end{proof}
It is now quite straightforward to prove Proposition \ref{bmsoln}.

\medskip
\noindent{\em Proof of Proposition \ref{bmsoln}.}

\ref{example:case:a}. Recall that in this case $K(z,t)=-2z/\alpha$.  Taking $h(z)=z/\alpha$
for $z\in \mathbb{R}$, we see that \ref{eq:supermtg:cts}-\ref{eq:supermtg:jump}-\ref{eq:mtg:cts}-\ref{eq:mtg:jump} are all satisfied with equality everywhere, without the qualifications involving the boundary $l$. Taking expectations in \eqref{doob-meyer}, and passing to the limit as $t\to\infty$ via dominated convergence, we see that $h$ is the optimal payoff, and \emph{any} control from $\CC$ realizes it. 

\ref{example:case:b}. Set  $l=0$  and $h(z)\eqdef\psi(\vert z\vert)$, where $\psi(z)\eqdef \frac{\gamma z+e^{-\gamma z}}{\alpha\gamma}$, $z\in \mathbb{R}$. 
 For such an $h$, $h$ is $C^2$, \ref{eq:supermtg:cts} is  satisfied with strict inequality on $z\in (-\infty,0)$; and  \ref{eq:mtg:cts} is satisfied on $[0,\infty)$. Then \ref{eq:mtg:jump} and \ref{eq:supermtg:jump}  are satisfied for $K$ of the form specified in \ref{example:case:b} on p.~\pageref{example:case:b} and the result follows by Lemma~\ref{verif}.  \qed

\section{Stopping times, stopped processes and natural filtrations at stopping times -- informational consistency}\label{appendix:stopped}
\subsection{The nature of information}
We now turn our attention to Question \ref{FiltQ}. We shall investigate: (i) the precise relationship between the sigma-fields of the stopped processes and the natural filtrations of the processes at these stopping times, and (ii) the nature of the stopping times of the processes and of the stopped processes, themselves. Here is an informal statement of the kind of results that we will establish ($\FF^X$ denotes the natural filtration of a process $X$):

\begin{quote}
If $X$ is a process, and $S$ a time, then $S$ is a stopping time of $\FF^X$, if and only if it is a stopping time of $\FF^{X^S}$. When this is so, then $\FF^X_S=\sigma(X^S)$. In particular, if $X$ and $Y$ are two processes, and $S$ is a stopping time of either $\FF^X$ or of $\FF^Y$, with $X^S=Y^S$, then $S$ is a stopping time of $\FF^X$ and $\FF^Y$ both, moreover $\FF^X_S=\sigma(X^S)=\sigma(Y^S)=\FF^Y_S$. Further, if $U\leq V$ are two stopping times of $\FF^X$, $X$ again being a process, then $\sigma(X^U)=\FF^X_U\subset \FF^X_V= \sigma(X^V)$.
\end{quote}

We will perform this study of the nature of information generated by processes in the  `measure-theoretic' and then the `probabilistic' setting. This will mirror the parallel development of the two frameworks for stochastic control from the preceding sections. 

The main findings of this section are as follows: 
\begin{itemize}
\item in the `measure-theoretic' case, Lemma~\ref{lemma:denumerable}, Proposition~\ref{lemma:continuous}, Theorem~\ref{theorem:galmarino}, Theorem~\ref{theorem:observational_consistency} and Proposition~\ref{proposition:information_increases};
\item  for the case with completions, Corollaries~\ref{corollary:observational_consistency_completions_discrete} and~\ref{corollary:stopping_times_completions} (in discrete time) and Proposition~\ref{proposition:represent_predictable}, Corollaries~\ref{corollary:continuous_stopped_filtration_sigma_field_of_stopped},~\ref{corollary:observational_consistency_completions} and~\ref{proposition:information_increases_completions} (in continuous time). 
\end{itemize}
We have already referenced many of these results in the preceding sections.

It emerges that everything that intuitively \emph{ought} to hold, \emph{does} hold, \emph{if} either the time domain is discrete, or else the underlying space is Blackwell (and, when dealing with completions, the stopping times are predictable and the filtration quasi-left-continuous; but see the negative results of Examples~\ref{example:completions_one} and~\ref{example:completions_two}). While we have not been able to drop the ``Blackwell assumption'', we believe many of the results should still hold true under weaker conditions -- this remains open.


\subsection{The `measure-theoretic' case}\label{section:the_abstract_case}
We begin with some relevant definitions. Indeed, the definitions to be introduced presently are (mostly) standard  \cite[passim]{kallenberg,dellacheriemeyer}: we provide their definitions explicitly in order to avoid any ambiguity, to fix notation, and to recall some measure-theoretic facts along the way. 

Throughout the remainder of this section $T=\mathbb{N}_0$ or $T=[0,\infty)$, $\Omega$ is a set and $(E,\Epsilon)$ is a measurable space.  When $t\in T=\mathbb{N}_0$, $[0,t]$ denotes the set $\{0,\ldots,t\}$. A time (on $\Omega$) is a map $\Omega\to S\cup \{\infty\}$. Recall also that $2^X$ denotes the power set of a set $X$ and $\sigma(f)$ denotes the $\sigma$-field generated by a map $f$ (the measurable structure on the codomain being understood from context). 

By a \textbf{process} (on $\Omega$, with time domain $T$ and values in $E$), we mean a collection $X=(X_t)_{t\in T}$ of functions from $\Omega$ into $E$. With $\FF^X_t\eqdef \sigma(X_s:s\in [0,t])$ for $t\in T$,  $\FF^X\eqdef (\FF^X_t)_{t\in T}$ is the \textbf{natural filtration} of $X$. Remark that, for every $t\in T$, $\FF^X_t=\sigma(X\vert_{[0,t]})$, where for $\omega\in \Omega$, $X\vert_{[0,t]}(\omega)\eqdef (X_s(\omega))_{s\in [0,t]}$; $X\vert_{[0,t]}:\Omega\to E^{[0,t]}$ is an $\FF^X_t$/$\Epsilon^{\otimes [0,t]}$-measurable map. For $\omega\in \Omega$, the $\omega$-\textbf{sample path} of $X$, $X(\omega)$, is the function $T\ni t\mapsto X_t(\omega)\in E$.  In this sense $X$ may of course be viewed as an $\FF^X_\infty$/$\Epsilon^{\otimes T}$-measurable map, indeed $\FF_\infty^X=\sigma(X)$. Then $\Im X$ will denote the range (image) of the function $X:\Omega\to E^T$.  Henceforth, unless otherwise made clear, we will consider a process as a map from $\Omega$ into $E^T$. 

If $S:\Omega\to T\cup \{\infty\}$ is a time, the \textbf{stopped process} $X^S$ is defined by $$X^S_t(\omega)\eqdef X_{S(\omega)\land t}(\omega), \quad \omega\in \Omega,\,t\in T.$$ If further $\GG$ is a \textbf{filtration} (i.e. a nondecreasing family of $\sigma$-fields indexed by $T$)  on $\Omega$ and  $S$ is \textbf{$\GG$-stopping time} (i.e. $\{S\leq t\}\in \GG_t$ for all $t\in T$), then 
$$\GG_S\eqdef \{A\in \GG_\infty: A\cap \{S\leq t\}\in \GG_t\text{ for each }t\in T\}$$
is the \textbf{filtration $\GG$ at (the time) $S$}.  Note that, if $T=\mathbb{N}_0$, if $X$ is a $\GG$-adapted process and if $S$ is a $\GG$-stopping time, then  $X^S$ is 
automatically adapted to the stopped filtration $(\GG_{n\land S})_{n\in\mathbb{N}_0}$. For, if $n\in \mathbb{N}_0$ and 
$Z\in \Epsilon$, then $(X^S_{n})^{-1}(Z)=\left(\cup_{m=0}^nX_{m}^{-1}(Z)\cap \{S=m\}\right)\cup  
\left(X_{n}^{-1}(Z)\cap \{n<S\}\right)\in \GG_{S\land n}$.\refstepcounter{dummy} \label{measurability_of_stopped} On 
the other hand, in continuous time, when $T=[0,\infty)$, if $X$ is $\GG$-progressively measurable and if $S$ is a $\GG$-stopping time, then $X^S$ is also adapted to the stopped filtration $(\GG_{t\land S})_{t\in [0, \infty)}$ (and is $\GG$-progressively measurable) \cite[p. 9, Proposition~2.18]{karatzasshreve}. Note also that every right- or left-continuous 
metric space-valued $\GG$-adapted process is automatically $\GG$-progressively measurable (limits of metric space-valued measurable functions being measurable).

Next, for a $\sigma$-field $\FF$ on $\Omega$, define an equivalence relation $\sim$ on $\Omega$, via $$(\omega\sim \omega')\overset{\mathrm{def}}{\Leftrightarrow}(\text{for all }A\in \FF,\mathbbm{1}_A(\omega)=\mathbbm{1}_A(\omega')).$$ Then define the \textbf{atoms} of $(\Omega,\FF)$ to be the equivalence classes of $\Omega$ under $\sim$. The \textbf{Hausdorff space} of $(\Omega,\FF)$ is the corresponding quotient space under $\sim$.

The measurable space $(\Omega,\FF)$ is said to be 
\begin{itemize}
\item[(i)]\textbf{separable} or \textbf{countably generated}, when it admits a countable generating set; 
\item[(ii)]\textbf{Hausdorff}, or \textbf{separated}, when its atoms are the singletons of  $\Omega$ \cite[p. 10]{dellacheriemeyer}; \item[and]
\item[(iii)]\textbf{Blackwell} when its associated Hausdorff space is Souslin  \cite[p. 50, III.24]{dellacheriemeyer}.
\end{itemize}
Furthermore, a Souslin space is a measurable space, which is Borel isomorphic to a Souslin topological space. The latter in turn is a Hausdorff topological space, which is also a continuous image of  a Polish space (i.e. of a completely metrizable separable topological space). Every Souslin measurable space is necessarily separable and separated. \cite[p. 46, III.16; p. 76, III.67]{dellacheriemeyer} For a measurable space, clearly being Souslin is equivalent to being simultaneously Blackwell and Hausdorff. 

The key result for us is Blackwell's Theorem \cite[p. 51 Theorem~III.26]{dellacheriemeyer} (repeated here for the reader's convenience -- we shall use it repeatedly):
\begin{theorem}[Blackwell's Theorem]\label{Black} Let $(\Omega,\FF)$ be a Blackwell space, $\GG$ a sub-$\sigma$-field of $\FF$ and $\SS$ a separable sub-$\sigma$-field of $\FF$. Then $\GG\subset \SS$, if and only if every atom of $\GG$ is a union of atoms of $\SS$. In particular, an $\FF$-measurable real function $g$ is $\SS$-measurable, if and only if $g$ is constant on every atom of $\SS$. \qed
\end{theorem}
Some elementary observations that we shall use without special reference are gathered in 
\begin{lemma}
\leavevmode
\begin{itemize}
\item[(i)]If $Y$ is a mapping from $A$ into some Hausdorff (respectively separable) measurable space $(B,\mathcal{B})$, then $Y$ is constant on the atoms of $\sigma(Y)$ (respectively $\sigma(Y)$ is separable). 

\item[(ii)]Conversely, if $Y$ is a surjective mapping from $A$ onto some measurable space $(B,\mathcal{B})$, constant on the atoms of $\sigma(Y)$, then $(B,\mathcal{B})$ is Hausdorff.

\item[(iii)]Any measurable subspace (with the trace $\sigma$-field) of a separable (respectively Hausdorff) space is separable (respectively Hausdorff). 
\item[(iv)]If $f:A\to (B,\BB)$ is any map into a measurable space, then the atoms of $\sigma(f)$ always `respect' the equivalence relation induced by $f$, i.e., for $\{\omega,\omega'\}\subset A$, if $f(\omega)=f(\omega')$, then $\omega$ and $\omega'$ belong to the same atom of $\sigma(f)$.
\end{itemize}
\end{lemma}
\begin{proof}
\leavevmode
\begin{itemize}
\item[(i)]Suppose that $(B,\mathcal{B})$ is Hausdorff and $A$ is an atom of $\sigma(Y)$ containing $\omega_1$ and $\omega_2$. If $Y(\omega_1)\ne Y(\omega_2)$, then there is a $W\in \mathcal{B}$ with $\mathbbm{1}_W(Y(\omega_1))\ne \mathbbm{1}_W(Y(\omega_2))$, hence $\mathbbm{1}_{Y^{-1}(W)}(\omega_1)\ne \mathbbm{1}_{Y^{-1}(W)}(\omega_2)$, a contradiction. 

\item[(ii)]If $\{b,b'\}\subset B$ and $\mathbbm{1}_W(b)=\mathbbm{1}_W(b')$ for all $W\in \mathcal{B}$, then if $\{a,a'\}\subset A$ are such that $Y(a)=b$ and $Y(a')=b'$, $\mathbbm{1}_{Z}(a)=\mathbbm{1}_Z(a')$ for all $Z\in \sigma(Y)$, so that $a$ and $a'$ belong to the same atom of $(A,\sigma(Y))$ and consequently $b=b'$. 
\end{itemize}
The rest of the claims are trivial.
\end{proof}

Now, a key result in this section will establish that, for a process $X$ and a stopping time $S$ thereof, $\sigma(X^S)=\FF^X_S$, i.e. that the initial structure (with respect to $\Epsilon^{\otimes T}$) of the stopped process coincides with the filtration of the process at the stopping time --  at least under suitable conditions. 

Our second lemma tells us that elements of $\FF^X_S$ are, possibly {non-measurable}, \emph{functions}  of the stopped process $X^S$. 

\begin{lemma}[Key lemma]\label{lemma:respect_stopping_time}
Let $X$ be a process (on $\Omega$, with time domain $T$ and values in $E$), $S$ an $\FF^X$-stopping time, $A\in \FF^X_S$. Then the following holds for every $\{\omega,\omega'\}\subset \Omega$: If $X_t(\omega)=X_t(\omega')$ for all $t\in T$ with $t\leq S(\omega)\land S(\omega')$, then 
$$S(\omega)=S(\omega'),\quad X^S(\omega)=X^S(\omega')\text{ and } \mathbbm{1}_A(\omega)=\mathbbm{1}_A(\omega').$$
\end{lemma}
\begin{remark}
This implies that if $A\in \FF^X_S$, then $\mathbbm{1}_A=F\circ X^S$ for some function $F:\Im X^S\to \{0,1\}$. Thus, if  $X^S:\Omega\to \Im X^S$ happens to be $\FF^X_S$/$2^{\Im X^S}$-measurable (this would typically be the case if the range of $X^S$ is denumerable, and the sample paths of $X$ are sufficiently regular), then it follows at once that $\sigma(X^S)=\FF^X_S$. 
\end{remark}
\begin{proof}
Define $t\eqdef S(\omega)\land S(\omega')$. If $t=\infty$, then clearly $S(\omega)=S(\omega')$. If not, then $\{S\leq t\}\in \FF^X_{t}$, so that there is a $U\in \Epsilon^{\otimes [0,t]}$ with $\{S\leq t\}=X\vert_{[0,t]}^{-1}(U)$. Then at least one of $\omega$ and $\omega'$ must belong to $\{S\leq t\}$, hence to $X\vert_{[0,t]}^{-1}(U)$. Consequently, since by assumption $X\vert_{[0,t]}(\omega)=X\vert_{[0,t]}(\omega')$, both do. It follows that $S(\omega)=S(\omega')$. In particular, $X^S(\omega)=X^S(\omega')$. 

Similarly, since $A\in \FF^X_S$, $A\cap \{S\leq t\}\in \FF^X_t$, so that there is a $U\in \Epsilon^{\otimes [0,t]}$ (respectively $U\in \Epsilon^{\otimes T}$), with $A\cap \{S\leq t\}=X\vert_{[0,t]}^{-1}(U)$ (respectively $A\cap \{S\leq t\}=X^{-1}(U)$), when $t<\infty$ (respectively $t=\infty$). Then $\mathbbm{1}_A(\omega)=\mathbbm{1}_{A\cap \{S\leq t\}}(\omega)=\mathbbm{1}_U(X\vert_{[0,t]}(\omega))=\mathbbm{1}_U(X\vert_{[0,t]}(\omega'))=\mathbbm{1}_{A\cap \{S\leq t\}}(\omega')=\mathbbm{1}_A(\omega')$ (respectively $\mathbbm{1}_A(\omega)=\mathbbm{1}_{A\cap \{S\leq t\}}(\omega)=\mathbbm{1}_U(X(\omega))=\mathbbm{1}_U(X(\omega'))=\mathbbm{1}_{A\cap \{S\leq t\}}(\omega')=\mathbbm{1}_A(\omega')$).
\end{proof}

Our second lemma deals with the discrete case. 

\begin{lemma}[Stopping times I]\label{lemma:denumerable}
Let $T=\mathbb{N}_0$ and $X$ be a process (on $\Omega$, with time domain $\mathbb{N}_0$ and values in $E$). For a  time $S:\Omega\to \mathbb{N}_0\cup \{\infty\}$ the following are equivalent: 
\begin{enumerate}
 \item $S$ is an $\FF^X$-stopping time. 
\item $S$ is an $\FF^{X^S}$-stopping time.
\end{enumerate}
\end{lemma}
\begin{proof}
Suppose $S$ is an $\FF^{X^S}$-stopping time. Let $n\in \mathbb{N}_0$. Then for each $m\in [0,n]$, $\{S\leq m\}\in \FF^{X^S}_m$, so there is an $E_m\in \Epsilon^{\otimes [0,m]}$ with $\{S\leq m\}=(X^{S}\vert_{[0,m]})^{-1}(E_m)$. Then $\{S=m\}\subset X\vert_{[0,m]}^{-1}(E_m)\subset \{S\leq m\}$. Consequently $\{S\leq n\}=\cup_{m\in [0,n]}X\vert_{[0,m]}^{-1}(E_m)\in \FF^{X}_n$. 

Conversely, suppose $S$ is an $\FF^X$-stopping time. Let $n\in \mathbb{N}_0$. For each $m\in [0,n]$, $\{S\leq m\}\in \FF^X_m$, hence there is an $E_m\in \Epsilon^{\otimes [0,m]}$ with $\{S\leq m\}=X\vert_{[0,m]}^{-1}(E_m)$. Then $\{S=m\}\subset (X^S\vert_{[0,m]})^{-1}(E_m)\subset \{S\leq m\}$. Consequently $\{S\leq n\}=\cup_{m\in [0,n]}(X^S\vert_{[0,m]})^{-1}(E_m)\in \FF^{X^S}_n$. 
\end{proof}
The next step establishes that members of $\FF^X_S$ are, in fact, \emph{measurable} functions of the stopped process $X^S$ -- at least under certain conditions (and always in the discrete case). 

\begin{proposition}\label{proposition:stopped}
Let $X$ be a process (on $\Omega$, with time domain $T$ and values in $E$), $S$ an $\FF^X$-stopping time. If any one of the three conditions below is fulfilled, then $\FF^X_S\subset \sigma(X^S)$ (where $X^S$ is viewed as assuming values in $(E^T,\Epsilon^{\otimes T})$). 
\begin{enumerate}[(M1)]
\item\label{cond:stopped:zero} $T=\mathbb{N}_0$. 
\item\label{cond:stopped:one} $\Im X^S\subset \Im X$.
\item\label{cond:stopped:two}
\begin{enumerate}
\item\label{cond:stopped:two:a} $(\Omega,\GG)$ is Blackwell for some $\sigma$-field $\GG\supset \FF^X_S\lor \sigma(X^S)$.
\item\label{cond:stopped:two:b} $\sigma(X^S)$ is separable (in particular, this obtains if $(\Im X^S,\Epsilon^{\otimes T}\vert_{\Im X^S})$ is separable).
\item\label{cond:stopped:two:c} $X^S$ is constant on the atoms of $\sigma(X^S)$, i.e. $(\Im X^S,\Epsilon^{\otimes T}\vert_{\Im X^S})$ is Hausdorff.
\end{enumerate}
\end{enumerate}
\end{proposition}

\begin{remark}\label{remark:important}
\leavevmode
\begin{enumerate}[(1)]
\item Condition \ref{cond:stopped:one} is clearly quite a strong assumption, but will typically be met when $X$ is the coordinate process on a canonical space. It is slightly weaker than condition (1.11) of \cite[Paragraph 1.1.3]{shiryaev} (see \cite[Paragraph~1.1.3, p.~10, Theorem~6]{shiryaev}).  
\item\label{remark:important:2} Condition \eqref{cond:stopped:two:b} is satisfied if there is a $\DD\subset E^T$, with $\Im X^S\subset \DD$, such that the trace $\sigma$-field $\Epsilon^{\otimes T}\vert_\DD$ is separable. For example (when $T=[0,\infty)$) this is the case if $E$ is a second countable (e.g. separable and  metrizable) topological space endowed with its (consequently separable) Borel $\sigma$-field, and the sample paths of $X^S$ are either all left- or all right-continuous (take $\DD$ to be all the left- or all the right-continuous paths from $E^{[0,\infty)}$). 
\item\label{remark:important:3} Finally, condition \eqref{cond:stopped:two:c} holds if $(E,\Epsilon)$ is Hausdorff and so, in particular, when the singletons of $E$ belong to $\Epsilon$. 
\end{enumerate}
\end{remark}
\begin{proof}
Assume first \ref{cond:stopped:zero}. Let $A\in \FF^X_S$ and $n\in \mathbb{N}_0$. Then $A\cap \{S=n\}\in \FF^X_n$, so $A\cap \{S=n\}=(X\vert_{[0,n]})^{-1}(Z)$ for some $Z\in \Epsilon^{\otimes [0,n]}$. But then  $A\cap \{S=n\}=(X^S\vert_{[0,n]})^{-1}(Z)\cap \{S=n\}$. Thanks to Lemma~\ref{lemma:denumerable}, $\{S=n\}\in \sigma(X^S)$. A similar argument deals with the case where $n=\infty$.

Assume now \ref{cond:stopped:one}. Let $A\in \FF^X_S$. Then $\mathbbm{1}_A=F\circ X$ for some $\Epsilon^{\otimes T}$/$2^{\{0,1\}}$-measurable mapping $F$. Since $\Im X^S\subset \Im X$, for any $\omega\in \Omega$, there is an $\omega'\in \Omega$ with $X(\omega')=X^S(\omega)$, and then thanks to Lemma~\ref{lemma:respect_stopping_time} $X^S(\omega')=X^S(\omega)$, moreover, $F\circ X^S(\omega)=F\circ X(\omega')=\mathbbm{1}_A(\omega')=\mathbbm{1}_A(\omega)$. It follows that $\mathbbm{1}_A=F\circ X^S$. 

Finally assume \ref{cond:stopped:two}. We apply Blackwell's Theorem. Specifically, thanks to \eqref{cond:stopped:two:a}  $(\Omega,\GG)$ is a Blackwell space and $\FF^X_S$ is a sub-$\sigma$-field of $\GG$; \eqref{cond:stopped:two:a} and \eqref{cond:stopped:two:b} imply that $\sigma(X^S)$ is a separable sub-$\sigma$-field of $\GG$. Finally, thanks to \eqref{cond:stopped:two:c} and Lemma~\ref{lemma:respect_stopping_time}, every atom (equivalently, every element) of $\FF^X_S$ is a union of atoms of $\sigma(X^S)$. It follows that $\FF^X_S\subset \sigma(X^S)$.
\end{proof}

The continuous-time analogue of Lemma~\ref{lemma:denumerable} is as follows: 
\begin{proposition}[Stopping times II]\label{lemma:continuous}
Let $T=[0,\infty)$ and $X$ be a process (on $\Omega$, with time domain $[0,\infty)$ and values in $E$), $S:\Omega\to [0,\infty]$ a time. Suppose:
\begin{enumerate}[(S1)]
\item $\sigma(X\vert_{[0,t]})$ and $\sigma(X^{S\land t})$ are separable, $(\Im X\vert_{[0,t]},\Epsilon^{\otimes [0,t]})$ and $(\Im X^{S\land t},\Epsilon^{\otimes T}\vert_{\Im X^{S\land t}})$ Hausdorff for each $t\in [0,\infty)$.
\item $X^S$ and $X$ are both measurable with respect to a Blackwell $\sigma$-field $\GG$ on $\Omega$. 
\end{enumerate}
Then the following are equivalent: 
\begin{enumerate}[(a)]
\item $S$ is an $\FF^X$-stopping time.
\item $S$ is an $\FF^{X^S}$-stopping time.
\end{enumerate}
\end{proposition}
\begin{proof}
Suppose first (b). Let $t\in [0,\infty)$. Then $\{S\leq t\}\in \FF^{X^S}_t$. But $\FF^{X^S}_t=\sigma(X^S\vert_{[0,t]})\subset \sigma(X\vert_{[0,t]})=\FF^X_t$. This follows from the fact that every atom of $\sigma(X^S\vert_{[0,t]})$ is a union of atoms of $ \sigma(X\vert_{[0,t]})$ (whence one can apply Blackwell's Theorem). To see this, note that if $\omega$ and $\omega'$ belong to the same atom of $\sigma(X\vert_{[0,t]})$, then $X\vert_{[0,t]}(\omega)=X\vert_{[0,t]}(\omega')$ (since  $(\Im X\vert_{[0,t]},\Epsilon^{\otimes [0,t]})$ is  Hausdorff). But then $X^S_s(\omega)=X^S_s(\omega')$ for all $s\in [0,(S(\omega)\land t)\land (S(\omega')\land t)]$, and so by Lemma~\ref{lemma:respect_stopping_time} (applied to the process $X^S$ and the stopping time $S\land t$ of $\FF^{X^S}$), $(X^S)^{S\land t}(\omega)=(X^S)^ {S\land t}(\omega')$, i.e. $X^S\vert_{[0,t]}(\omega)=X^S \vert_{[0,t]}(\omega')$. We conclude that $\omega$ and $\omega'$ belong to the same atom of $\sigma(X^S\vert_{[0,t]})$.

Conversely, assume (a).  Let $t\in [0,\infty)$. Then $\{S\leq t\}\in \FF^{X}_{S\land t}$, $S\land t$ is an $\FF^X$-stopping time and thanks to Proposition~\ref{proposition:stopped}, $\FF^{X}_{S\land t}\subset \sigma(X^{S\land t})=\sigma(X^S\vert_{[0,t]})$.
\end{proof}
We now give the main result of this subsection. As mentioned in the Introduction, it generalizes canonical space results already available in the literature.

\begin{theorem}[A generalized Galmarino's test]\label{theorem:galmarino}
Let $X$ be a process (on $\Omega$, with time domain $T$ and values in $E$), $S$ an $\FF^X$-stopping time. 
\begin{enumerate}[(1)]
\item If $T=\mathbb{N}_0$, then $\sigma(X^S)=\FF^X_S$. 
\item Moreover, if $X^S$ is $\FF^X_S$/$\Epsilon^{\otimes T}$-measurable (in particular, if it is adapted to the stopped filtration $(\FF^X_{t\land S})_{t\in T}$) and either one of the conditions:

\begin{enumerate}[(G1)]
\item\label{cond:stoppedd:one} $\Im X^S\subset \Im X$.
\item\label{cond:stoppedd:two}
\begin{enumerate}[(a)]
\item\label{cond:stoppedd:two:i} $(\Omega,\GG)$ is Blackwell for some $\sigma$-field $\GG\supset \FF^X_\infty$.
\item\label{cond:stoppedd:two:ii} $\sigma(X^S)$ is separable.
\item\label{cond:stoppedd:two:iii} $(\Im X^S,\Epsilon^{\otimes T}\vert_{\Im X^S})$ is Hausdorff.
\end{enumerate}
\end{enumerate}

is met, then the following statements are equivalent:
\begin{enumerate}[(i)]
\item\label{galmarino:one} $A\in \FF^X_S$.
\item\label{galmarino:two} $\mathbbm{1}_A$ is constant on every set on which $X^S$ is constant and $A\in \FF^X_\infty$.
\item\label{galmarino:three} $A\in \sigma(X^S)$. 
\end{enumerate}
\end{enumerate}
\end{theorem}
\begin{proof}
The first claim, which assumes $T=\mathbb{N}_0$, follows from Proposition~\ref{proposition:stopped} and the fact that automatically $X^S$ is $\FF^X_S$/$\Epsilon^{\otimes T}$-measurable in this case.

In general, implication \ref{galmarino:one}$\Rightarrow$\ref{galmarino:two} follows from Lemma~\ref{lemma:respect_stopping_time}. Implication \ref{galmarino:two}$\Rightarrow$\ref{galmarino:three} proceeds as follows. 

Suppose first  \ref{cond:stoppedd:one}. Let $\mathbbm{1}_A$ be constant on every set on which $X^S$ is constant, $A \in \FF^X_\infty$. Then  $\mathbbm{1}_A=F\circ X$ for some $\Epsilon^{\otimes T}$/$2^{\{0,1\}}$-measurable mapping $F$. Next, from $\Im X^S\subset \Im X$, for any $\omega\in \Omega$, there is an $\omega'\in \Omega$ with $X(\omega')=X^S(\omega)$, and then thanks to Lemma~\ref{lemma:respect_stopping_time}, $X^S(\omega')=X^S(\omega)$, so that by assumption, $\mathbbm{1}_A(\omega)=\mathbbm{1}_A(\omega')$, also. Moreover, $F\circ X^S(\omega)=F\circ X(\omega')=\mathbbm{1}_A(\omega')=\mathbbm{1}_A(\omega)$. It follows that $\mathbbm{1}_A=F\circ X^S$. 

Assume now \ref{cond:stoppedd:two}. Again apply Blackwell's Theorem. Specifically, on account of \ref{cond:stoppedd:two}\ref{cond:stoppedd:two:i}  $(\Omega,\GG)$ is a Blackwell space and $\FF^X_S$ is a sub-$\sigma$-field of $\GG$; on account of \ref{cond:stoppedd:two}\ref{cond:stoppedd:two:ii} and $\sigma(X^S)\subset \FF^X_S$, $\sigma(X^S)$ is a separable sub-$\sigma$-field of $\GG$. Finally, if $\mathbbm{1}_A$ is constant on every set on which $X^S$ is constant and $A\in \FF^X_\infty$, then $\mathbbm{1}_A$ is a $\GG$-measurable function (by \ref{cond:stoppedd:two}\ref{cond:stoppedd:two:i}), constant on every atom of $\sigma(X^S)$ (by \ref{cond:stoppedd:two}\ref{cond:stoppedd:two:iii}). It follows that $\mathbbm{1}_A$ is $\sigma(X^S)$-measurable. 

The implication  \ref{galmarino:three}$\Rightarrow$\ref{galmarino:one} is just one of the assumptions.
\end{proof}
As for our original motivation into this investigation, we obtain: 

\begin{corollary}[Observational consistency]\label{theorem:observational_consistency}
Let $X$ and $Y$ be two processes  (on $\Omega$, with time domain $T$ and values in $E$), $S$ an $\FF^X$ and an $\FF^Y$-stopping time. Suppose furthermore $X^S=Y^S$.  If any one of the conditions
\begin{enumerate}[(O1)]
\item\label{observable:consistency:zero} $T=\mathbb{N}_0$.
\item\label{observable:consistency:one} $\Im X=\Im Y$.
\item[or]
\item \label{observable:consistency:two}
\begin{enumerate}[(a)]
\item $(\Omega,\GG)$ (respectively $(\Omega,\HH)$) is Blackwell for some $\sigma$-field $\GG\supset \FF^X_\infty$ (respectively $\HH\supset \FF^Y_\infty$).
\item $\sigma(X^S)$ (respectively $\sigma(Y^S)$) is separable and contained in $\FF^X_S$ (respectively $\FF^Y_S$).
\item $(\Im X^S,\Epsilon^{\otimes T}\vert_{\Im X^S})$ and $(\Im Y^S,\Epsilon^{\otimes T}\vert_{\Im Y^S})$ are  Hausdorff.
\end{enumerate}
\end{enumerate}
is met, then $\FF^X_S=\FF^Y_S$.
\end{corollary}
\begin{remark}\label{remark:observational_consistency}
If $T=\mathbb{N}_0$, then rather than requiring that $S$ be a stopping time of both $\FF^X$ and $\FF^Y$, it follows from Lemma~\ref{lemma:denumerable} that it is sufficient  for $S$ to be a stopping time of just one of them.
The same is true when \ref{observable:consistency:two} obtains, as long as the conditions of Proposition~\ref{lemma:continuous} are met for the time $S$ and both the processes $X$ and $Y$.
\end{remark}
\begin{proof}
If \ref{observable:consistency:zero}  or \ref{observable:consistency:two} hold, then the claim follows immediately from Theorem~\ref{theorem:galmarino}.

If \ref{observable:consistency:one} holds, let $A\in \FF^X_S$, $t\in T$. Then $\mathbbm{1}_{A\cap \{S\leq t\}}=F\circ X\vert_{[0,t]}$ for some $\Epsilon^{\otimes [0,t]}$/$2^{\{0,1\}}$-measurable $F$. Moreover, for each $\omega\in \Omega$, there is an $\omega'\in \Omega$ with $X(\omega')=Y(\omega)$, hence $X(\omega)$ agrees with $Y(\omega)=X(\omega')$  on $T\cap [0,S(\omega)]$, and thus thanks to Lemma~\ref{lemma:respect_stopping_time}, $S(\omega)=S(\omega')$ and  $\mathbbm{1}_A(\omega)=\mathbbm{1}_A(\omega')$. This implies that $F\circ Y\vert_{[0,t]}(\omega)=F \circ X\vert_{[0,t]}(\omega')=\mathbbm{1}_{A\cap \{S\leq t\}}(\omega')=\mathbbm{1}_{A\cap \{S\leq t\}}(\omega)$, i.e. $\mathbbm{1}_{A\cap \{S\leq t\}}=F\circ Y\vert_{[0,t]}$. The case with $t=\infty$ is similar.
\end{proof}
We also have:

\begin{proposition}[Monotonicity of information]\label{proposition:information_increases}
Let $Z$ be a process (on $\Omega$, with time domain $T$ and values in $E$) and $U\leq V$ be two stopping times of $\FF^Z$. If either  
\begin{enumerate}[(N1)]
\item $T=\mathbb{N}_0$ 
\item[or else]
\item all of
\begin{enumerate}
\item $(\Omega, \GG)$ is Blackwell for some $\sigma$-field $\GG\supset \sigma(Z^V)\lor \sigma(Z^U)$.
 \item $(\Im Z^V,\Epsilon^{\otimes T}\vert_{\Im Z^V})$  is Hausdorff.
 \item[and]
\item $\sigma(Z^V)$ is separable, 
\end{enumerate}
\end{enumerate}
 then $\sigma(Z^U)\subset \sigma(Z^V)$. 
\end{proposition}
\begin{proof}
In the discrete case the result follows at once from Theorem~\ref{theorem:galmarino}. In the continuous case, we claim that the assumptions imply that every atom of $\sigma(Z^U)$ is a union of the atoms of $\sigma(Z^V)$.
To see this, let $\omega$ and $\omega'$ belong to the same atom of $\sigma(Z^V)$; then, since $(\Im Z^V,\Epsilon^{\otimes T}\vert_{\Im Z^V})$  is Hausdorff, $Z^V(\omega)=Z^V(\omega')$, hence  Lemma~\ref{lemma:respect_stopping_time} implies that $V(\omega)=V(\omega')$ and $U(\omega)=U(\omega')$, and so \emph{a fortiori} $Z^U(\omega)=Z^U(\omega')$, which implies that $\omega$ and $\omega'$ belong to the same atom of $\sigma(Z^U)$. Now apply Blackwell's Theorem.
\end{proof}

\subsection{The completed case}\label{section:case_with_completions}
We have studied in the previous subsection natural filtrations, 
now we turn our attention to their completions. For a filtration $\GG$ on $\Omega$ and a \emph{complete} probability measure $\PP$, whose domain includes $\GG_\infty$, we denote by $\overline{\GG}^{\PP}$ the \textbf{completed filtration} given by $\overline{\GG}^{\PP}_t \eqdef \GG_t\lor \NN$, $\NN$ being the $\PP$-null sets; likewise if the domain of $\PP$ includes a $\sigma$-field $\AA$ on $\Omega$, then $\overline{\AA}^{\PP}\eqdef \AA\lor \NN$ ($=\{A'\text{ from the domain of }\PP\text{ such that }\exists A\in \AA\text{ with }\PP(A\triangle A')=0\}$). For any other unexplained notation, that we shall use, we refer the reader to the beginning of section~\ref{section:the_abstract_case}. 
Recall in particular that $T=\mathbb{N}_0$ or $T=[0,\infty)$, $\Omega$ is a set and $(E,\Epsilon)$ is a measurable space.

First, all is well in the discrete case:
\begin{lemma}\label{lemma:observational_consistency_denumerable_completions}
Let $T=\mathbb{N}_0$ and  $\GG$ be a filtration on $\Omega$. Furthermore, let $\PP$ be a complete probability measure on $\Omega$, whose domain includes $\GG_\infty$ and let $S$ be a $\overline{\GG}^{\PP}$-stopping time. Then:
\begin{enumerate}
\item $S$ is $\PP$-a.s. equal to a stopping time $S'$ of $\GG$; 
\item for any $\GG$-stopping time $U$ which is $\PP$-a.s. equal to $S$, $\overline{\GG_{U}}^{\PP}=\overline{\GG}^{\PP}_S$;
\item[and]
\item if $U$ is any random time, $\PP$-a.s equal to $S$, then it is, in fact, a $\overline{\GG}^{\PP}$-stopping time, and $\overline{\GG}^{\PP}_S=\overline{\GG}^{\PP}_U$.
\end{enumerate}
\end{lemma}
\begin{proof}
\leavevmode
\begin{enumerate}
\item For each $n\in \mathbb{N}_0$, we may find an $A_n\in \GG_n$, such that $\{S=n\}=A_n$, $\PP$-a.s. Then $S'\eqdef \sum_{n\in \mathbb{N}_0}n\mathbbm{1}_{A_n}+\infty\mathbbm{1}_{\Omega\backslash \cup_{m\in \mathbb{N}_0}A_m}$ is a $\GG$-stopping time, $\PP$-a.s. equal to $S$. 

\item Now let $U$ be any $\GG$-stopping time, $\PP$-a.s. equal to $S$. To show $\overline{\GG_U}^{\PP}\subset \overline{\GG}^{\PP}_S$, it suffices to note that (i) $\NN$ is contained in $\overline{\GG}^{\PP}_S$ and (ii) $\GG_U\subset \overline{\GG}^{\PP}_S$. Conversely, if $A\in \overline{\GG}^{\PP}_S$, then for each $n\in \mathbb{N}_0\cup \{\infty\}$, $A\cap \{S=n\}=B_n$, $\PP$-a.s., for some $B_n\in \GG_n$, and hence the event $\cup_{n\in \mathbb{N}_0\cup \{\infty\}}B_n\cap \{U=n\}$ belongs to $\GG_U$, and is $\PP$-a.s. equal to $A$.

\item Finally, let $U$ be a random time,  $\PP$-a.s equal to $S$. For each $n\in \mathbb{N}_0\cup \{\infty\}$, there is then a $C_n\in \GG_n$ with $\{U=n\}=\{S=n\}=C_n$, $\PP$-a.s., so $U$ is a $\overline{\GG}^{\PP}$-stopping time. It follows that we can find $U'$, a $\GG$-stopping time, $\PP$-a.s. equal to $U$, hence $S$, and thus with $\overline{\GG}^{\PP}_S=\overline{\GG_{U'}}^{\PP}=\overline{\GG}^{\PP}_U$.
\end{enumerate}
\end{proof}

\begin{corollary}\label{corollary:observational_consistency_completions_discrete}
Suppose that $T=\mathbb{N}_0$; $X$ and $Y$ are processes (on $\Omega$, with time domain $\mathbb{N}_0$ and values in $E$); $\PP^X$ and $\PP^Y$ are complete probability measures on $\Omega$ whose domains contain $\FF^X_\infty$ and $\FF^Y_\infty$, respectively, and with the \emph{same null sets}. Suppose furthermore $S$ is an $\overline{\FF^X}^{\PP^X}$ and an $\overline{\FF^Y}^{\PP^Y}$  stopping time, with $X^S=Y^S$, $\PP^X$ and $\PP^Y$-a.s. Then $\overline{\FF^X}^{\PP^X}_S=\overline{\sigma(X^S)}^{\PP^X}=\overline{\sigma(Y^S)}^{\PP^Y}=\overline{\FF^Y}^{\PP^Y}_S$.
\end{corollary}
 \begin{proof}
From Lemma~\ref{lemma:observational_consistency_denumerable_completions} we can find stopping times $U$ and $V$ of $\FF^X$ and $\FF^Y$, respectively, both $\PP^X$ and $\PP ^Y$-a.s. equal to $S$. The event $\{X^{U}=Y^{V}\}$ is $\PP^X$ and $\PP^Y$-almost certain. It then follows, from Theorem~\ref{theorem:galmarino} and Lemma~\ref{lemma:observational_consistency_denumerable_completions} again, that $\overline{\FF^X}^{\PP^X}_S=\overline{\FF^X_U}^{\PP^X}=\overline{\sigma(X^U)}^{\PP^X}
=\overline{\sigma(Y^V)}^{\PP^Y}=\overline{\FF^Y_V}^{\PP^Y}
=\overline{\FF^Y}^{\PP^Y}_S$, as desired.
\end{proof}

\begin{corollary}\label{corollary:stopping_times_completions}
Let $T=\mathbb{N}_0$, $X$ be a process (on $\Omega$, with time domain $\mathbb{N}_0$ and values in $E$), $\PP$ be a complete probability measure on $\Omega$ whose domain contains $\FF^X_\infty\lor \FF^{X^S}_\infty$ and $S:\Omega\to T\cup \{\infty\}$ be a random time. Then the following are equivalent: 
\begin{enumerate}
\item $S$ is an $\overline{\FF^X}^{\PP}$-stopping time. 
\item $S$ is an $\overline{\FF^{X^S}}^{\PP}$-stopping time. 
\end{enumerate}
\end{corollary}
\begin{proof}
That (1)$\Rightarrow$(2) is clear from Lemma~\ref{lemma:observational_consistency_denumerable_completions}, Lemma~\ref{lemma:denumerable} and the fact that two processes, which are versions of each other, generate the same filtration, up to null sets. For the reverse implication, one can proceed as in  the relevant part of the proof of Lemma~\ref{lemma:denumerable}, adding $\PP$-a.s. qualifiers as appropriate.
\end{proof}
The continuous-time case is much more involved. Indeed, we have the following significant negative results:
\begin{example}\label{example:completions_one}Consider the following setup.
$\Omega=(0,\infty)\times \{0,1\}$; $\FF$ is the product of the Lebesgue $\sigma$-field on $(0,\infty)$ and of the power set on $\{0,1\}$;  $\PP=\Exp(1)\times \Unif(\{0,1\})$ is the product law on $\FF$ (which is complete); $e$ (respectively $I$) is the projection onto the first (respectively second) coordinate. The process $N$ is given by  $N_t=(t-e)\mathbbm{1}_{[0,t]}(e)I$  for $t\in [0,\infty)$ (starting at zero, the process $N$ departs from zero at time $e$ with unit positive drift, or remains at zero for all times, with equal probability, independently of $e$). Its completed natural filtration, $\overline{\FF^N}^{\PP}$, is already right-continuous.

For, if $t\in [0,\infty)$, $\overline{\FF^N}^{\PP}_{t+}=\overline{\FF^N_{t+} }^\PP$; so let $A\in \FF^N_{t+}$, we show $A\in \overline{\FF^N_t}^\PP$. (i) $A\cap \{e=t\}$ is $\PP$-negligible. (ii) Clearly $A=N^{-1}(G)$, for some measurable $G\subset \mathbb{R}^{[0,\infty)}$. Then define for each natural $n\geq 1/t$ (when $t>0$), $L_n:\mathbb{R}^{[0,t]}\to \mathbb{R}^{[0,\infty)}$, by demanding
$$L_n(\omega)(u)=
\begin{cases}
\omega(u), &\text{ for } u\leq t\\
\omega(t)+(u-t)\frac{\omega(t)-\omega(t-1/n)}{1/n}, &\text{ for }u>t
\end{cases}
$$ ($u\in [0,\infty)$, $\omega\in \mathbb{R}^{[0,t]}$), a measurable mapping. It follows that for $t>0$, for each natural $n \geq 1/t$, $N^{-1}(G)\cap \{e\leq t-1/n\}=N\vert_{[0,t]}^{-1}(L_n^{-1}(G))\cap \{e\leq t-1/n\}\in \FF^N_t$. (iii) For each natural $n$, $A\cap \{e>t+1/n\}=N\vert_{[0,t+1/n]}^{-1}(G_n)\cap \{e>t+1/n\}$ for some measurable $G_n\subset \mathbb{R}^{[0,t+1/n]}$, so $A\cap \{e>t+1/n\}$ is $\emptyset$ or $\{e>t+1/n\}$ according as $0$ is an element of $G_n$ or not (note this is a ``monotone'' condition, in the sense that as soon as we once get a non-empty set for some natural $n$, we subsequently get $\{e>t+1/m\}$ for all natural $m\geq n$). It follows that $A\cap \{e>t\}=\cup_{n\in \mathbb{N}}(A\cap \{e>t+1/n\})\in \{\emptyset,\{e>t\}\}\subset \FF^N_t$.

Further, let $U$ be the first entrance time of the process $N$ to $(0,\infty)$. By the D\'ebut Theorem, this is a stopping time of $\overline{\FF^N}^\PP$, but it is not $\PP$-a.s. equal to any stopping time of $\FF^N$ at all.

For suppose that it were $\PP$-a.s. equal to a stopping time $V$ of $\FF^N$. Then there would be a set $\Omega'$, belonging to $\FF$, of full $\PP$-measure, and such that $V=U$ on $\Omega'$. Tracing everything ($\FF$, $\PP$, $N$, $a$, $e$, $V$) onto $\Omega'$, we would obtain ($\FF'$, $\PP'$, $N'$, $a'$, $e'$, $V'$), with (i) $V'$ equal to the first entrance time of $N'$ to $(0,\infty)$ and (ii) $V'$ a stopping time of $\FF^{N'}$, the natural filtration of $N'$. Note that $N'_t=a'(t-e')\mathbbm{1}_{[0,t]}(e')$, $t\in[ 0,\infty)$. Now take $\{\omega,\omega'\}\subset \Omega'$ with $a(\omega)=1$, $a(\omega')=0$, denote $t\eqdef e(\omega)$. Then $N'\vert_{[0,t]}(\omega)=N'\vert_{[0,t]}(\omega')$, so $\omega$ and $\omega'$ should belong to the same atom of $\FF^{N'}_t$; yet $\{V'\leq t\}\in \FF^{N'}_t$, with $\mathbbm{1}_{\{V'\leq t\}}(\omega)=1$ and $\mathbbm{1}_{\{V'\leq t\}}(\omega')=0$, a contradiction.

Moreover, $\overline{\FF^N}^{\PP}_U\ne \overline{\sigma(N^U)}^{\PP}$, since the event $A\eqdef \{U<\infty\}=\{a=1\}$ that $N$ ever assumes a positive drift belongs to  $\overline{\FF^N}^{\PP}_U$ (which fact is clear), but not to $\overline{\sigma(N^U)}^{\PP}=\overline{\sigma(0)}^\PP$, the trivial $\sigma$-field (it is also obvious; $\PP(a=1)=1/2\notin \{0,1\}$).\finish
\end{example}

\begin{example}\label{example:completions_two}
It is even worse. Let $\Omega=(0,\infty)\times \{-2,-1,0\}$ be endowed with the law $\PP=\Exp(1)\times \Unif(\{-2,-1,0\})$, defined on the tensor product of the Lebesgue $\sigma$-field on $(0,\infty)$ and the power set of $\{-2,-1,0\}$. Denote by $e$, respectively $I$, the projection onto the first, respectively second, coordinate. Define the process $X$ by $X_t\eqdef (t-e)\mathbbm{1}_{[0,t]}(e)I$, $t\in [0,\infty)$, and the process $Y_t\eqdef (-1)(t-e)\mathbbm{1}_{[0,t]}(e)\mathbbm {1}_{\{-1,-2\}}\circ I $, $t\in [0,\infty)$. The completed natural filtrations of $X$ and $Y$ are already right-continuous. The first entrance time $S$ of $X$ into $(-\infty,0)$ is equal to the first entrance time of $Y$ into $(-\infty,0)$, and this is a stopping time of $\overline{\FF^X}^\PP$ as it is of $\overline{\FF^Y}^\PP$ (but not of $\FF^X$ and not of $\FF^Y$). Moreover, $X^S=0=Y^S$. 

Consider now the event $A\eqdef \{I=-1\}$. It is clear that $A\in \overline{\FF^X}^{\PP}_S$. However, $A\notin \overline{\FF^Y}^{\PP}_S$. For, assuming the converse, we should have, $\PP$-a.s., $\mathbbm{1}_{A\cap \{S\leq 1\}}=F\circ Y\vert_{[0,1]}$ for some, measurable, $F$. In particular, since $A\cap \{S\leq 1\}$ has positive probability, there should be an $\omega\in A\cap \{S\leq 1\}$ with $F(Y\vert_{[0,1]}(\omega))=1$. But also the event $\{I=-2\}\cap \{S\leq 1\}$ has positive probability and is disjoint from $A\cap \{S\leq 1\}$, so there should be an $\omega'\in \{I=-2\}\cap \{S\leq 1\}$ having $F(Y\vert_{[0,1]}(\omega'))=0$. A contradiction, since nevertheless $Y\vert_{[0,1]}(\omega')=Y\vert_{[0,1]}(\omega)$. \finish
\end{example}
The problem here is that in completing the natural filtration the apparently innocuous operation of adding all the events negligible under $\PP$ is done uncountably many times (once for every deterministic time). In particular, this does not correspond to a single completion of the sigma-field generated by the stopped process: completions are not always harmless.

Furthermore, it is not clear to us what a sensible \emph{direct} `probabilistic' analogue of Lemma~\ref{lemma:respect_stopping_time} should be, never mind how to go about proving one.

However, the situation is not entirely bleak, since positive results can be obtained at least for foretellable/predictable stopping times of quasi-left-continuous filtrations  -- as in the case of discrete time -- by an \emph{indirect} method; reducing the `probabilistic' to the `measure-theoretic' case. We use here the terminology of \cite[pp.~127 and 137, Definitions~IV.69,~IV.70 and~IV.84]{dellacheriemeyer}, i.e. given a filtration $\GG$ and a probability measure $\QQ$ on $\Omega$, whose domain includes $\GG_\infty$: 
\begin{definition}
a random time $S:\Omega\to [0,\infty]$ is \textbf{predictable} relative to $\GG$ if the stochastic interval $\llbracket T,\infty\llbracket$ is predictable. It is \textbf{$\QQ$-foretellable} relative to $\GG$ if there exists a $\QQ$-a.s. nondecreasing sequence $(S_n)_{n\geq 1}$ of $\GG$-stopping times with $S_n\leq S$, $\QQ$-a.s for all $n\geq 1$ and such that, again $\QQ$-a.s., $$\lim_{n\to \infty}S_n=S, S_n<S\text{ for all }n\text{ on }\{S>0\};$$ \textbf{foretellable}, if the a.s. qualifications can be omitted. Finally, $\GG$ is \textbf{quasi-left-continuous} if $\GG_T=\GG_{T-}$ for all predictable times $T$ of $\GG$.
\end{definition}
Note that (i) the property of predictability is invariant under passage to the right-continuous augmentation of a filtration, and (ii) in a $\PP$-complete filtration ($\PP$ itself assumed complete), the notions of predictable, foretellable and $\PP$-foretellable stopping times coincide \cite[ p. 127, IV.70; p. 128, Theorem~IV.71  and  p. 132, Theorem~IV.77]{dellacheriemeyer}.

The following is now a complement to \cite[p.~120, Theorem~IV.59  and  p.~133, Theorem~IV.78]{dellacheriemeyer}  \cite[p. 5,  Lemma~1.19]{jacod}, and an analogue of the discrete statement of Lemma~\ref{lemma:observational_consistency_denumerable_completions}:

\begin{proposition}\label{proposition:represent_predictable}
Let $T=[0,\infty)$, $\GG$ be a filtration on $\Omega$. Let $\PP$ be a complete probability measure on $\Omega$, whose domain includes $\GG_\infty$ and $S$ be a predictable stopping time relative to $\overline{\GG}^{\PP}$ that we assume is quasi-left-continuous. Then:
\begin{enumerate}[(F1)]
\item\label{fore1} $S$ is $\PP$-a.s. equal to a predictable stopping time $W$ of $\GG$. 
\item\label{fore2} Moreover, if $U$ is any $\GG$-stopping time, $\PP$-a.s. equal to $S$, then $\overline{\GG}^\PP_S=\overline{\GG_U}^\PP$. 
\item Finally, if $S'$ is another random time, $\PP$-a.s equal to $S$, then it is a predictable $\overline{\GG}^{\PP}$-stopping time, and $\overline{\GG}^{\PP}_S=\overline{\GG}^{\PP}_{S'}$.
\end{enumerate}
\end{proposition}
\begin{proof}
\leavevmode
\begin{enumerate}[(F1)]
\item This is contained in \cite[p.~133, Theorem~IV.78]{dellacheriemeyer}.
\item Now let $U$ be any $\GG$-stopping time, $\PP$-a.s. equal to $S$. The inclusion  $\overline{\GG}^\PP_S \supset \overline{\GG_U}^\PP$ is obvious. Then take $A\in \overline{\GG}^\PP_S$. Since $A\in \overline{\GG}^\PP_\infty=\overline{\GG_\infty}^\PP$, there is an $A'\in \GG_\infty$, such that $A'=A$, $\PP$-a.s. Furthermore, thanks to quasi-left-continuity, $S_A=S\mathbbm{1}_A+\infty\mathbbm{1}_{\Omega\backslash A}$ is again  $\overline{\GG}^{\PP}$-predictable \cite[p. 418, Proposition~22.15]{kallenberg}. Hence, by \ref{fore1}, there exists $V$, a $\GG$-stopping time, with $V=S_A$, $\PP$-a.s. So, 
$$A=(A'\cap \{U=\infty\})\cup \{V=U<\infty\}\in \GG_U,\text{  $\PP$-a.s.}$$ 
\item Finally let $S'$ be a random time, $\PP$-a.s. equal to $S$. Clearly, it is a predictable  $\overline{\GG}^\PP$-stopping time. Moreover, by \ref{fore1}, we can find $W$, a $\GG$-stopping time, $\PP$-a.s. equal to $S$ and  hence  (by hypothesis) $S'$. It follows from \ref{fore2} that $\overline{\GG}^{\PP}_S=\overline{\GG_W}^\PP=\overline{\GG}^{\PP}_{S'}$.
\end{enumerate} 
\end{proof}
From this we can obtain easily some useful  counterparts to the findings of section~\ref{section:the_abstract_case} in the continuous case: Corollaries~\ref{corollary:continuous_stopped_filtration_sigma_field_of_stopped}, \ref{corollary:observational_consistency_completions} and \ref{proposition:information_increases_completions} below. They can be applied to completions of Blackwell spaces in conjunction with (in this order) (i) the fact that a standard Borel space
-valued random element measurable with respect to the completed domain of a probability measure $\QQ$ is $\overline{\QQ}$-a.s. equal to a random element measurable with respect to the uncompleted domain of $\QQ$ ($\overline{\QQ}$ being the completion of $\QQ$) \cite[p.~13, Lemma~1.25]{kallenberg} and (ii) part \ref{fore1}of Proposition~\ref{proposition:represent_predictable}.

\begin{corollary}\label{corollary:continuous_stopped_filtration_sigma_field_of_stopped}
Let $T=[0,\infty)$, $Z$ be a process (on $\Omega$, with time domain $T$ and values in $E$), $\PP$ be a complete probability measure on $\Omega$, whose domain includes $\FF^Z_\infty$, and $W$ be an $\overline{\FF^Z}^\PP$-predictable stopping time with $\overline{\FF^Z}^\PP$ quasi-left-continuous. If, for some process $X$ $\PP$-indistinguishable from $Z$ and a stopping time $S$ of $\FF^X$, with $S=W$, $\PP$-a.s., Condition \ref{cond:stopped:two} holds, then $\overline{\FF^Z}^\PP_W\subset \overline{\sigma(Z^W)}^\PP$. 
\end{corollary}
\begin{remark}
The reverse inclusion $\overline{\sigma(Z^W)}\subset \overline{\FF^Z}^\PP_W$ is usually trivial. 
\end{remark}
\begin{proof}
According to Proposition~\ref{proposition:stopped}, $\FF^X_S\subset \sigma(X^S)$. Also $\overline{\FF^X}^\PP=\overline{\FF^Z}^\PP$ and $\overline{\sigma(X^S)}^\PP=\overline{\sigma(Z^W)}^\PP$. Taking completions in $\FF^X_S\subset \sigma(X^S)$, by applying \ref{fore2} of Proposition~\ref{proposition:represent_predictable} to the stopping time $W$ of $\overline{\FF^X}^\PP$ (which is $\PP$-a.s. equal to the stopping time $S$ of $\FF^X$), we obtain:  $$\overline{\FF^Z}^\PP_W=\overline{\FF^X}^\PP_W=\overline{\FF^X_S}^\PP\subset \overline{\sigma(X^S)}^\PP=\overline{\sigma(Z^W)}^\PP,$$ as desired.
\end{proof}

\begin{corollary}\label{corollary:observational_consistency_completions}
Let $T=[0,\infty)$; let $Z$ and $W$ be two processes  (on $\Omega$, with time domain $[0,\infty)$ and values in $E$); $\PP^Z$ and $\PP^W$ be probability measures on $\Omega$, with the same null sets, and whose domains include $\FF^Z_\infty$ and $\FF^W_\infty$, respectively; $V$ a predictable $\overline{\FF^Z}^{\PP^Z}$- and  $\overline{\FF^W}^{\PP^W}$-stopping time, with $\overline{\FF^Z}^{\PP^Z}$  and $\overline{\FF^W}^{\PP^W}$ quasi-left-continuous. Suppose furthermore $Z^V=W^V$, $\PP^Z$ and $\PP^W$-a.s.  

If there exist two processes $X$ and $Y$, indistinguishable from $Z$ and $W$, respectively, and stopping times $S$ and $U$ of $\FF^X$ and $\FF^Y$, respectively, with $S=U=V$, $\PP^Z$ and $\PP^W$-a.s. and such that the pairs $(X,S)$ and $(Y,U)$ each satisfy Condition \ref{cond:stopped:two},
then $\overline{\FF^Z}^{\PP^Z}_V=\overline{\sigma(Z^V)}^{\PP^Z}=
\overline{\sigma(W^V)}^{\PP^W}=\overline{\FF^W}^{\PP^W}_V$.
\end{corollary}
\begin{proof}
The claim follows from Corollary~\ref{corollary:continuous_stopped_filtration_sigma_field_of_stopped}, and the fact that again 
$\sigma(X^S)\subset \FF^X_S$ implies $\overline{\sigma(Z^V)}^{\PP^Z}\subset \overline{\FF^Z}^{\PP^Z}_V$; with a comparable statement holding  for $W$.
\end{proof}

\begin{corollary}\label{proposition:information_increases_completions}
Let $T=[0,\infty)$;  $X$ be a process (on $\Omega$, with time domain $[0,\infty)$ and values in $E$); $\PP$ a complete probability measure on $\Omega$, whose domain includes $\FF^X_\infty$; $S$ and $\stop$ two predictable stopping times of $\overline{\FF^X}^\PP$ with $S\leq \stop$ and $\overline{\FF^X}^\PP$  quasi-left-continuous. Let $U$ and $V$ be two stopping times of the natural filtration of a process $Z$, $\PP$-indistinguishable from $X$, $\PP$-a.s. equal to $S$ and $\stop$, respectively, with $U\leq V$, and such that: 
\begin{enumerate}[(Q1)]
\item\label{order1} $(\Omega, \GG)$ is Blackwell for some $\sigma$-field $\GG\supset \sigma(Z^V)\lor \sigma(Z^U)$, 
\item\label{order2}  $(\Im Z^V,\Epsilon^{\otimes T}\vert_{\Im Z^V})$  is Hausdorff 
\item[and]
\item\label{order3}  $\sigma(Z^V)$ is separable. 
\end{enumerate}
Then $\overline{\sigma(X^S)}^\PP\subset \overline{\sigma(X^\stop)}^\PP$. 
\end{corollary}
\begin{remark}
Proposition~\ref{proposition:represent_predictable}\ref{fore1} establishes the existence of $U$ and $V'$, stopping times of $\FF^Z$, $\PP$-a.s. equal to $S$ and $\stop$, respectively, with $U\leq V'$, $\PP$-a.s. 
Defining $V\eqdef V'\mathbbm{1}(U\leq V')+U\mathbbm{1}(U>V')$, $V$ is also a stopping time of $\FF^Z$, $\PP$-a.s. equal to $\stop$, and \emph{it} satisfies $U\leq V$ with certainty. The question is whether these stopping times satisfy conditions \ref{order1}--\ref{order3}. 

\end{remark}
\begin{proof}
We  see  that $\overline{\sigma(X^S)}^\PP=\overline{\sigma(Z^U)}^\PP$ and $\overline{\sigma(X^\stop)}^\PP=\overline{\sigma(Z^V)}^\PP$. Applying Proposition~\ref{proposition:information_increases} gives the result. 
\end{proof}

We are not able to provide a 
\emph{useful}  counterpart to Proposition~\ref{lemma:continuous}.  In the notation of  Corollary~\ref{proposition:information_increases_completions}, Proposition~\ref{proposition:represent_predictable} \emph{does say} that, given a predictable stopping time $P$ of $\overline{\FF^{X^P}}^\PP$, there is a predictable stopping time $U$ of $\FF^{Z^P}$, $\PP$-a.s. equal to $P$. But this \emph{does not \emph{a priori} say} that $U$ is a stopping time of $\FF^{Z^U}$, so one cannot directly apply Proposition~\ref{lemma:continuous}.
\bibliographystyle{plain}
\bibliography{Biblio_stochastic_control}

\begin{thebibliography}{10}

\bibitem{bensoussan}
A.~Bensoussan.
\newblock {\em Stochastic Control of Partially Observable Systems}.
\newblock Cambridge University Press, 2004.

\bibitem{alain}
A.~Bensoussan.
\newblock {\em Stochastic Control by Functional Analysis Methods}.
\newblock Studies in Mathematics and its Applications. Elsevier Science, 2011.

\bibitem{blumenthal-getoor}
R.~M. Blumenthal and R.~K. Getoor.
\newblock A theorem on stopping times.
\newblock {\em The Annals of Mathematical Statistics}, 35(3):1348--1350, 1964.

\bibitem{brzezniak}
Z.~Brze\'zniak and S.~Peszat.
\newblock {Space-time continuous solutions to SPDE's driven by a homogeneous
  Wiener process}.
\newblock {\em Studia Mathematica}, 137(3):261--299, 1999.

\bibitem{davis-varaiya}
M.~H.~A. Davis and P.~Varaiya.
\newblock {Dynamic Programming Conditions for Partially Observable Stochastic
  Systems}.
\newblock {\em SIAM Journal on Control}, 11(2):226--261, 1973.

\bibitem{dellacheriemeyer}
C.~Dellacherie and P.~A. Meyer.
\newblock {\em Probabilities and Potential A}.
\newblock North-Holland Mathematics Studies. Herman Paris, 1978.

\bibitem{donsk}
M.~D. Donsker and S.~R.~S. Varadhan.
\newblock Asymptotics for the {W}iener {S}ausage.
\newblock {\em {Communications in Pure and Applied Mathematics}},
  28(4):525–565, 1975.

\bibitem{elkaroui}
N.~El~Karoui.
\newblock {Les Aspects Probabilistes Du Contr\^ole Stochastique}.
\newblock In P.~L. Hennequin, editor, {\em Ecole d'Et\'{e} de Probabilit\'{e}s
  de Saint-Flour IX-1979}, volume 876 of {\em Lecture Notes in Mathematics},
  chapter~2, pages 73--238. Springer Berlin / Heidelberg, 1981.

\bibitem{elkaroui_jeanblanc}
N.~El~Karoui, D.~Nguyen, and M.~Jeanblanc-Picqué.
\newblock {Existence of an Optimal Markovian Filter for the Control under
  Partial Observations}.
\newblock {\em SIAM Journal on Control and Optimization}, 26(5):1025--1061,
  1988.

\bibitem{fleming_pardoux}
W.~Fleming and E.~Pardoux.
\newblock Optimal {C}ontrol for {P}artially {O}bserved {D}iffusions.
\newblock {\em SIAM Journal on Control and Optimization}, 20(2):261--285, 1982.

\bibitem{fleming-rishel}
W.H. Fleming and R.W. Rishel.
\newblock {\em Deterministic and Stochastic Optimal Control}.
\newblock Applications of mathematics. Springer-Verlag, 1975.

\bibitem{fryer}
R.~G. Fryer and P.~Harms.
\newblock {Two-Armed Restless Bandits with Imperfect Information: Stochastic
  Control and Indexability}.
\newblock NBER Working Papers 19043, National Bureau of Economic Research, Inc,
  2013.

\bibitem{gihman}
I.~I. Gihman and A.~V. Skorohod.
\newblock {\em Controlled Stochastic Processes}.
\newblock Springer New York, 2012.

\bibitem{jacod}
J.~Jacod and A.~N. Shiryaev.
\newblock {\em Limit {T}heorems for {S}tochastic {P}rocesses}.
\newblock Grundlehren der mathematischen Wissenschaften. Springer-Verlag,
  Berlin Heidelberg, 2003.

\bibitem{olav}
O.~Kallenberg.
\newblock {\em Lectures on {R}andom {M}easures}.
\newblock Institute of Statistics Mimeo series. University of Goteberg, 1974.

\bibitem{kallenberg}
O.~Kallenberg.
\newblock {\em Foundations of Modern Probability}.
\newblock Probability and Its Applications. Springer, New York Berlin
  Heidelberg, 1997.

\bibitem{karatzasshreve}
I.~Karatzas and S.~E. Shreve.
\newblock {\em Brownian Motion and Stochastic Calculus}.
\newblock Graduate Texts in Mathematics. Springer New York, 1991.

\bibitem{protter}
P.~E. Protter.
\newblock {\em Stochastic Integration and Differential Equations: Version 2.1}.
\newblock Applications of mathematics. U.S. Government Printing Office, 2004.

\bibitem{rishel}
R.~Rishel.
\newblock Necessary and {S}ufficient {D}ynamic {P}rogramming {C}onditions for
  {C}ontinous {T}ime {S}tochastic {O}ptimal {C}ontrol.
\newblock {\em SIAM Journal on Control}, 8(4):559--571, 1970.

\bibitem{rogers}
L.~C.~G. Rogers and D.~Williams.
\newblock {\em {Diffusions, Markov Processes, and Martingales: Volume 1,
  Foundations}}.
\newblock Cambridge Mathematical Library. Cambridge University Press, 2000.

\bibitem{shiryaev}
A.~N. Shiryaev.
\newblock {\em Optimal Stopping Rules}.
\newblock Stochastic Modelling and Applied Probability. Springer Berlin
  Heidelberg, 2007.

\bibitem{soner}
H.~M. Soner and N.~Touzi.
\newblock Dynamic {P}rogramming for {S}tochastic {T}arget {P}roblems and
  {G}eometric {F}lows.
\newblock {\em Journal of the European Mathematical Society}, 4(3):201--236,
  2002.

\bibitem{striebel}
C.~Striebel.
\newblock Martingale {C}onditions for the {O}ptimal {C}ontrol of {C}ontinuous
  {T}ime {S}tochastic {S}ystems.
\newblock {\em Stochastic Processes and their Applications}, 18(2):329 -- 347,
  1984.

\bibitem{varadhan}
D.~W. Stroock and S.~R.~S. Varadhan.
\newblock {\em Multidimensional Diffusion Processes}.
\newblock Grundlehren der mathematischen Wissenschaften. Springer Berlin
  Heidelberg, 1997.

\bibitem{varaiya}
P.~Varaiya.
\newblock {Optimal Control of a Partially Observed Stochastic System}.
\newblock In J.~B. Keller and H.~P. McKean, editors, {\em Stochastic
  Differential Equations}, SIAM-AMS proceedings, pages 173--188. American
  Mathematical Society, 1973.

\bibitem{wonham}
W.~Wonham.
\newblock On the {S}eparation {T}heorem of {S}tochastic {C}ontrol.
\newblock {\em SIAM Journal on Control}, 6(2):312--326, 1968.

\bibitem{yong}
J.~Yong and X.~Y. Zhou.
\newblock {\em Stochastic Controls: Hamiltonian Systems and HJB Equations}.
\newblock Stochastic Modelling and Applied Probability. Springer New York,
  1999.

\bibitem{yuksel}
S.~Y{\"u}ksel and T.~Ba\c{s}ar.
\newblock {\em Stochastic Networked Control Systems}.
\newblock Systems \& Control: Foundations \& Applications. Birkh{\"a}user
  Basel, 2013.

\end{thebibliography}
\appendix
\goodbreak
\section{Miscellaneous technical results}
Throughout this appendix $(\Omega,\FF,\PP)$ is a probability space; $\EE$ denotes expectation with respect to $\PP$. 

\begin{lemma}[On conditioning]\label{conditioninglemma}
Let $X:\Omega\to [-\infty,\infty]$ be a random variable, and $\GG_{i}\subset \FF$, $i=1,2$, be two sub-$\sigma$-fields of $\FF$ agreeing when traced on $A\in \GG_1\cap \GG_2$. Then, $\PP$-a.s. on $A$, $\EE[X\vert \GG_1]=\EE[X\vert \GG_2]$, whenever $X$ has a $\PP$-integrable positive or negative part. 
\end{lemma}
\begin{proof}
$\mathbbm{1}_AZ$ is $\GG_2$-measurable, for any $Z$ $\GG_1$-measurable, by an approximation argument. Then, $\PP$-a.s., $\mathbbm{1}_A\EE[X\vert \GG_1]=\EE[\mathbbm{1}_AX\vert \GG_2]$, by the very definition of conditional expectation.
\end{proof}
Recall Definition~\ref{accinf}.
\begin{lemma}\label{lemma:endoftime}
Suppose that $\HH$ is a (respectively a $\PP$-complete) 
filtration on $\Omega$ in discrete ($T=\mathbb{N}_0$) or continuous ($T=[0,\infty)$) time and that a sequence $(S_n)_{n\in\mathbb{N}}$ of its stopping times accesses infinity pointwise (respectively $\PP$-a.s.) on $A\subset \Omega$. Then $\HH_\infty\vert_A=\lor_{n\in \mathbb{N}}\HH_{S_n}\vert_A$. 
\end{lemma}

\begin{remark}
Note that for any $\mathcal{L}\subset 2^\Omega$ and $A\subset \Omega$, $\sigma_\Omega(\mathcal{L})\vert_A=\sigma_A(\mathcal{L}\vert_A)$, so there is no ambiguity in writing $\lor_{n\in \mathbb{N}}\HH_{S_n}\vert_A$. 
\end{remark}
\begin{proof}
The inclusion $\HH_\infty\vert_A\supset\lor_{n\in \mathbb{N}}\HH_{S_n}\vert_A$ is clear. To establish the reverse inclusion let $t\in T$ and $B\in \HH_t$. \color{black}Then (respectively $\PP$-a.s.) $B\cap A=\cup_{n=1}^\infty (B\cap \{S_n\geq t\})\cap A$ with $B\cap \{S_n\geq t\}\in \HH_{S_n}$.
\end{proof}
The following lemma uses the notation, and is to be understood in the context, of sections~\ref{section:setting} and~\ref{section:cond_payoff_sytem}. 
\begin{lemma}\label{lemma:accessinginfinity}
Let $\{c,d\}\subset \CC$ and $A\subset \Omega$.  Suppose that $(\SS_n)_{n\in\mathbb{N}}$ is a sequence in $\Gtimes$ accessing infinity \{a.s.\} on $A$ for the controls $c$ and $d$, and for which $c\sim_{\SS_n}d$ for each $n\in \mathbb{N}$. Then $\GG^c_\infty\vert_A=\GG^d_\infty\vert_A$. 

If further, $A\in \GG^{c}_{\SS^c_n}$ for all $n\in \mathbb{N}$, and $(\SS^h_n(\omega))_{n\in \mathbb{N}}$ is nondecreasing for \{$\PP^h$-almost\} every $\omega\in A$, each $h\in \{c,d\}$, then $\PP^c\vert_{\GG^c_\infty}$ and $\PP^d\vert_{\GG^d_\infty}$ agree when traced on $A$. 
\end{lemma}
\begin{remark}
We mean to address here abstractly the situation when the two controls $c$ and $d$ agree for all times on $A$. 
\end{remark}
\begin{proof}
By stability under stopping, certainly $\PP^c\vert_{\mathcal{G}^c_{\SS_n^c}}$ agrees with $\PP^d\vert_{\mathcal{G}^d_{\SS_n^d}}$ for each $n\in \mathbb{N}$, while $(\SS_n^c=\SS_n^d)_{n\in \mathbb{N}}$ accesses infinity \{$\PP^c$-a.s. and $\PP^d$-a.s.\} on $A$. Then apply Lemma~\ref{lemma:endoftime} to obtain $\GG^c_\infty\vert_A=\sigma_A(\cup_{n\in \mathbb{N}}\mathcal{G}^c_{\SS_n^c}\vert_A)=\sigma_A(\cup_{n\in \mathbb{N}}\mathcal{G}^d_{\SS_n^d}\vert_A)=\GG^d_\infty\vert_A.$ If, moreover $A\in \GG^{c}_{\SS^c_n}$ for all $n\in \mathbb{N}$, then the traces of $\PP^c$ and $\PP^d$ on $A$ agree on $\cup_{n\in \mathbb{N}}\mathcal{G}^c_{\SS_n^c}\vert_A$. Provided in addition $(\SS_n^c)_{n\in\mathbb{N}}$ is \{$\PP^c$-a.s.\} nondecreasing on $A$, the latter union is a $\pi$-system (as a nondecreasing union of $\sigma$-fields, so even an algebra) on $A$. This, coupled with the fact that probability measures  which agree on a generating $\pi$-system are equal, yields the second claim.
\end{proof}

\begin{lemma}[Generalised conditional Fatou and Beppo-Levi lemmas]\label{lemma:beppoleviFatou}
Let $\GG\subset \FF$ be a sub-$\sigma$-field and $(f_n)_{n\geq 1}$ a sequence of $[-\infty,\infty]$-valued random elements, whose negative parts are dominated $\PP$-a.s by a single $\PP$-integrable random variable. Then, $\PP$-a.s., $$\EE[\liminf_{n\to\infty}f_n\vert\GG]\leq \liminf_{n\to\infty}\EE[f_n\vert\GG].$$ If, moreover, $(f_n)_{n\geq 1}$ is $\PP$-a.s. nondecreasing, then, $\PP$-a.s., $$\EE[\lim_{n\to\infty}f_n\vert \GG]=\lim_{n\to\infty}\EE[f_n\vert\GG].$$
\end{lemma}
\begin{proof}
Just apply the conditional version of Fatou's Lemma (respectively of the Beppo-Levi Lemma) to the $\PP$-a.s. nonnegative (respectively nonnegative nondecreasing) sequence $f_n+g$ where $g$ is the $\PP$-integrable random variable which $\PP$-a.s. dominates the negative parts of the $f_n$. Then use linearity and subtract the $\PP$-a.s. finite quantity $\EE[g\vert \GG]$. 
\end{proof}
The following is a slight generalization of \cite[Theorem~A2]{striebel}. 

\begin{lemma}[Essential supremum and the upwards lattice property]\label{lemma:esssup}
Let $\GG\subset \FF$ be a sub-$\sigma$-field and $X=(X_\lambda)_{\lambda\in \Lambda}$ a collection of $[-\infty,\infty]$-valued random variables with integrable negative parts. Assume furthermore that for each $\{\epsilon,M\}\subset (0,\infty)$, $X$  has the ``$(\epsilon,M)$-upwards lattice property'', i.e. for all $\{\lambda,\lambda'\}\subset \Lambda$, one can find a $\lambda''\in \Lambda$ with $X_{\lambda''}\geq (M\land X_\lambda)\lor (M\land X_{\lambda'})-\epsilon$ $\PP$-a.s. Then, $\PP$-a.s.,
\begin{equation}\label{eq:essential_supremum}
\EE[\PPesssup_{\lambda\in\Lambda}X_\lambda\vert\mathcal{G}]=\PPesssup_{\lambda\in\Lambda}\EE[X_\lambda\vert \GG],
\end{equation}
where on the right-hand side the essential supremum may of course equally well be taken with respect to the measure $\PP\vert_\GG$.
\end{lemma}
\begin{proof}
It is assumed without loss of generality that $\Lambda\ne\emptyset$, whence remark that $\PPesssup_{\lambda\in\Lambda}X_\lambda$ has an integrable negative part. Then the inequality 
$$\EE[\PPesssup_{\lambda\in\Lambda}X_\lambda\vert\mathcal{G}]\geq\PPesssup_{\lambda\in\Lambda}\EE[X_\lambda\vert \GG]
$$
is immediate.

Conversely, we show first that it is sufficient to establish the reverse inequality for each truncated $(X_\lambda\land N)_{\lambda\in\Lambda}$ family, as $N$ runs over $\mathbb{N}$. Indeed, suppose we have $\PP$-a.s. $$\EE[\PPesssup_{\lambda\in\Lambda}X_\lambda\land N\vert\mathcal{G}]\leq \PPesssup_{\lambda\in\Lambda}\EE[X_\lambda\land N\vert \GG]$$ for all $N\in \mathbb{N}$. Then \emph{a fortiori} $\PP$-a.s. for all $N\in \mathbb{N}$, $$\EE[\PPesssup_{\lambda\in\Lambda}X_\lambda\land N\vert\mathcal{G}]\leq \PPesssup_{\lambda\in\Lambda}\EE[X_\lambda\vert \GG]$$ and generalised conditional monotone convergence (Lemma~\ref{lemma:beppoleviFatou}) allows to pass to the limit: 
$$\EE[\lim_{N\to\infty}\PPesssup_{\lambda\in\Lambda}X_\lambda\land N\vert\mathcal{G}]\leq \PPesssup_{\lambda\in\Lambda}\EE[X_\lambda\vert \GG]$$ $\PP$-a.s. But clearly, $\PP$-a.s.,  $\lim_{N\to\infty}\PPesssup_{\lambda\in\Lambda}X_\lambda\land N\geq \PPesssup_{\lambda\in\Lambda}X_\lambda$, since for all $\lambda\in \Lambda$, we have, $\PP$-a.s., $X_\lambda\leq \lim_{N\to\infty}X_\lambda\land N\leq \lim_{N\to\infty}\PPesssup_{\mu\in\Lambda}X_\mu\land N$. 

Thus it will indeed be sufficient to establish the ``$\leq$-inequality'' in \eqref{eq:essential_supremum} for the truncated families, and so it is assumed without loss of generality (take $M=N$) that $X$ enjoys, for each $\epsilon\in (0,\infty)$, the ``$\epsilon$-upwards lattice property'': for all $\{\lambda,\lambda'\}\subset \Lambda$, one can find a $\lambda''\in \Lambda$ with $X_{\lambda''}\geq X_\lambda\lor X_{\lambda'}-\epsilon$ $\PP$-a.s.

Then take $(\lambda_n)_{n\geq 1}\subset \Lambda$ such that, $\PP$-a.s., $\PPesssup_{\lambda\in\Lambda}X_\lambda=\sup_{n\geq 1}X_{\lambda_n}$ and fix $\delta>0$. Recursively define $(\lambda_n')_{n\geq 1}\subset\Lambda$ so that,  $X_{\lambda_1'}=X_{\lambda_1}$ while for $n\in \mathbb{N}$, $\PP$-a.s., $X_{\lambda_{n+1}'}\geq X_{\lambda_n'} \lor X_{\lambda_{n+1}}-\delta/2^{n}$. Prove by induction that $\PP$-a.s. for all $n\in\mathbb{N}$, $X_{\lambda_n'}\geq \max_{1\leq k\leq n}(X_{\lambda_k}-\sum_{l=1}^{n-1}\delta/2^l)$, so that $\liminf_{n\to\infty}X_{\lambda'_n}\geq \sup_{n\in \mathbb{N}}X_{\lambda_n}-\delta$, $\PP$-a.s. Note next that the negative parts of $(X_{\lambda'_n})_{n\in \mathbb{N}}$ are dominated $\PP$-a.s. by a single $\PP$-integrable random variable. By the generalised conditional Fatou's lemma (Lemma~\ref{lemma:beppoleviFatou}) we therefore obtain, $\PP$-a.s., $\PPesssup_{\lambda\in\Lambda}\EE[X_\lambda\vert \GG]\geq \liminf_{n\to\infty} \EE[X_{\lambda'_n}\vert\GG]\geq \EE[\liminf_{n\to\infty}X_{\lambda'_n}\vert\GG]\geq \EE[\PPesssup_{\lambda\in\Lambda}X_\lambda\vert\mathcal{G}]-\delta$. Finally, let $\delta$ descend to $0$ (over some sequence descending to $0$). 
\end{proof}

\begin{lemma}\label{lemma:adap-predictable}
\leavevmode
\begin{enumerate}[(i)]
\item Let $\AA$ and $\BB$ be two $\sigma$-fields on $\Omega$, $F\subset \Omega$, $V:\Omega\to [-\infty,\infty]$, such that $\AA\vert_F=\BB\vert_{F}$ with $F\in \AA\cap \BB$. Then 
 $V\mathbbm{1}_F$ is $\BB$-measurable if $V$ is $\AA$-measurable.
\item Let $\AA$ and $\BB$ be two filtrations in continuous ($T=[0,\infty)$) 
 or discrete ($T=\mathbb{N}_0$)  time on $\Omega$, $V$ be a real-valued process on $\Omega$ with time domain $T$, $P$ be a map $P:\Omega\to T\cup\{\infty\}$, and $A\subset \Omega$. Assume that for all $t\in T$, 
\begin{enumerate}[(I)]
\item $\AA_t\vert_{\{t\leq P\}}=\BB_t\vert_{\{t\leq P\}}$, $\AA_t\vert_{\{t> P\}\cap A}=\BB_t\vert_{\{t> P\}\cap A}$ 

and 
\item  the events $\{t\leq P\}$, $A\cap \{P<t\}$ belong to $\BB_t\cap \AA_t$. 
\end{enumerate}
Then  $V\mathbbm{1}_{\llbracket 0,P\rrbracket}$ and  $V\mathbbm{1}_{\llparenthesis P,\infty\rrparenthesis}\mathbbm{1}_A$ are $\BB$-adapted (respectively, predictable) if $V$ is $\AA$-adapted (respectively, predictable). 
\end{enumerate}
\end{lemma}
\begin{proof}
The first part is clear. The second part in discrete time follows at once. In continuous time, adaptedness also follows at once from the first part. For predictability, one notes that the class of processes $V$ for which $V\mathbbm{1}_{\llbracket 0,P\rrbracket}$ and  $V\mathbbm{1}_{\llparenthesis P,\infty\rrparenthesis}\mathbbm{1}_A$  are $\BB$-predictable  contains the multiplicative class of all left-continuous $\AA$-adapted processes. The Functional Monotone Class Theorem allows us to extend this claim to all $\AA$-predictable processes. 
\end{proof}

\end{document}